\documentclass[a4paper]{amsart}
\usepackage[margin=3.5cm]{geometry}
\usepackage[utf8]{inputenc}
\usepackage{amsmath,amssymb,amsthm,amscd,amsfonts,comment,mathtools}
\usepackage{tikz,tikz-cd}
\usepackage{stmaryrd} 
\usepackage{hyperref}

\newcommand{\Z}{{\mathbb Z}}                   
\newcommand{\R}{{\mathbb R}}                   
\newcommand{\C}{{\mathbb C}}                   
\newcommand{\ZZ}{{\mathbb Z}}                   
\newcommand{\RR}{{\mathbb R}}                   
\newcommand{\CC}{{\mathbb C}}                   

\newcommand{\CZ}{{\mathcal Z}}

\newcommand{\adam}{\color{blue}}

\newtheorem{theorem}{Theorem}[section]  
\newtheorem{proposition}[theorem]{Proposition}
\newtheorem{corollary}[theorem]{Corollary}
\newtheorem{lemma}[theorem]{Lemma}
\newtheorem{conjecture}[theorem]{Conjecture}

\theoremstyle{definition}
\newtheorem{definition}[theorem]{Definition}

\theoremstyle{remark}
\newtheorem{remark}[theorem]{Remark}

\newtheorem{example}[theorem]{Example}

\newcommand{\SL}{\operatorname{SL}}

\def\u{u}  
\def\pt{pt}

\renewcommand{\pmod}[1]{\,(\mathrm{mod}\,\,#1)}

\title{Fixed point counts and motivic invariants of bow varieties of affine type A}

\author{\ \hskip 1.85 true cm \'Ad\'am Gyenge}
\address{Budapest University of Technology and Economics, 
Department of Algebra and Geometry, 
Institute of Mathematics, 
M\H{u}egyetem rakpart 3, H-1111, 
Budapest, Hungary}
\email{gyenge.adam@ttk.bme.hu}

\author{Rich\'ard Rim\'anyi}
\address{University of North Carolina, Chapel Hill, NC, USA}
\email{rimanyi@email.unc.edu}

\makeatletter
\let\@wraptoccontribs\wraptoccontribs
\makeatother
\contrib[\newline with an appendix by]{Gergely Harcos}
\address{Alfr\'ed R\'enyi Institute of Mathematics, Budapest, Hungary}
\email{gharcos@renyi.hu}

\subjclass[2020]{Primary 14D20; Secondary 	14D21, 16G20}
\keywords{bow variety, torus fixed point, generating series, equivariant K-theory}

\begin{document}

\begin{abstract}
We compute the equivariant K-theory of torus fixed points of Cherkis bow varieties of affine type A. We deduce formulas
for the generating series of the Euler numbers of these varieties and observe their modularity in certain cases. We also obtain refined formulas on the motivic level for a class of bow varieties strictly containing Nakajima quiver varieties. These series hence generalise results of Nakajima-Yoshioka. As a special case, we obtain formulas for certain Zastava spaces. 
We define a parabolic analogue of Nekrasov's partition function and find an equation relating it to the classical partition function.
\end{abstract}

\begin{abstract}
    The collection of spaces called Cherkis bow varieties include Nakajima quiver varieties, and have favorable properties; for example they are closed for 3d mirror symmetry. In this paper we prove formulas for the generating series of the Euler numbers of bow varieties and observe their modularity, in some cases. We develop the combinatorics to describe the equivariant K-theory class of the tangent spaces at torus fixed points. We identify a sub-class of bow varieties that contain the quiver varieties and Zastava spaces that share favorable properties with them, such as no odd cohomology. For such quiver-like varieties we prove refined motivic enumeration series generalise results of Nakajima-Yoshioka. We define a parabolic analogue of Nekrasov's partition function and find an equation relating it to the classical partition function.
\end{abstract}

\begin{abstract}
    The class of spaces known as Cherkis bow varieties includes Nakajima quiver varieties and exhibits several favorable properties, such as being closed under 3d mirror symmetry. In this paper, we derive formulas for the generating series of the Euler numbers of bow varieties and explore their modularity in certain cases. We prove combinatorial expressions for the equivariant K-theory class of the tangent spaces at torus fixed points. We identify a subclass of bow varieties that includes quiver varieties and Zastava spaces, sharing advantageous features like the absence of odd cohomology. For these quiver-like varieties, we refine the motivic enumeration series, extending results of Nakajima and Yoshioka. Furthermore, we introduce a parabolic variant of Nekrasov's partition function and show how it relates to the classical partition function.
\end{abstract}

\maketitle

\section{Introduction}

Cherkis bow varieties of affine type $\hat{A}$ were originally introduced as an ADHM-type construction to describe moduli spaces of Yang-Mills instantons on ALF spaces, such as the Taub-NUT space, that are equivariant under a cyclic $\mathbb{Z}/m\mathbb{Z}$-action \cite{cherkis2011instantons}. Subsequently, Nakajima and Takayama provided an algebro-geometric description of these bow varieties using quivers \cite{nakajima2017cherkis}. These varieties are algebraic with hyper-Kähler structures on their regular loci. In the algebraic setting, they are conjectured to correspond to moduli spaces of equivariant parabolic framed torsion-free sheaves on the projective plane blown up at a point.

Bow varieties are associated to a combinatorial object called brane diagram, which consists of $m$ NS5 branes and $n$ D5 branes, along with a set of integers referred to as dimensions and local charges. The varieties are naturally equipped with a torus action that has finitely many fixed points. Combinatorial descriptions of these fixed points were provided in \cite{nakajima2018towards} using generalized Maya diagrams, and in \cite[Appendix~A]{rimanyi2020bow} using tie diagrams.

The enumerative geometry of bow varieties was initiated in \cite{rimanyi2020bow}, where the focus was on characteristic classes (called stable envelopes) in the finite type A case. The expected 3d mirror symmetry property of these characteristic classes was later proven in \cite{BottaRimanyi}. While type~A bow varieties are significant, the truly important moduli spaces, such as Hilbert schemes or moduli spaces of torsion-free framed sheaves, are bow varieties of affine type~$\hat{A}$.

\smallskip

In this paper, we study bow varieties of affine type $\hat{A}$ and aim to perform instanton counting by computing various motivic invariants of naturally occurring infinite sequences of bow varieties. Our first result concerns the Euler characteristic as a motivic invariant, and we obtain the following theorem.

\begin{theorem}[{Theorem~\ref{thm:eulerchargen}}] 
\label{thm:main1}
Let $e \in \ZZ^{n}$ and $f \in \ZZ^{m}$ be fixed vectors. Together with an additional dimension parameter $d \in \ZZ_{\geq 0}$ they determine a bow variety $\mathcal{M}(d,e,f)$. We have
\label{thm:eulerchargen_intro}
\[\sum_{d=0}^{\infty} \chi(\mathcal{M}(d,e,f))q^d =  \left( \sum_{\substack{c \in \ZZ^{n \times m} \\ \sum_j c_{ij}=e_i \\ \sum_{i}c_{ij}=f_j }}
q^{\sum_{i,j} c_{ij}(c_{ij}-1)/2}
\right) \cdot \prod_{l=1}^{\infty} \left(\frac{1}{1-q^l}\right)^{nm}. \]
\end{theorem}
The appearance of contingency tables $c$, which are integer matrices with prescribed row and column sums, is noteworthy. 
These tables are fundamental objects in both enumerative combinatorics \cite{barvinok} and statistics \cite{lauritzen}. In the summation, the $q$-exponent is quadratic, resembling a squared distance. This observation identifies $Z(q)$ as a (shifted) lattice theta function, aligning with the physical expectation that instanton counts exhibit modular properties. In the Appendix, authored by G.~Harcos, we establish the specific modularity for a particular example.

\smallskip

In the remainder of the paper, we begin the study of Bia\l ynicki-Birula decompositions of bow varieties and the resulting motivic enumerations, leading to a $(q,t)$-generalization of the $q$-count in Theorem~\ref{thm:main1}.

Such enumerations, containing substantial birational \cite{bellamy2020birational} or motivic information \cite{nakajima2005instanton, gyenge2018euler}, are already known for the subclass of Nakajima quiver varieties. However, extending these results to bow varieties reveals unexpected complexities. Bow varieties are not as `perfect' as Nakajima quiver varieties---or, depending on one’s perspective, they possess richer motivic structures. For instance, unlike quiver varieties, the union of Bia\l ynicki-Birula cells does not generally cover the entire variety $\mathcal{M}(d,e,f)$. Additionally, in contrast to quiver varieties, bow varieties may exhibit odd cohomology groups.

Hence, as a preparation for refined counting formulas, we give detailed geometric and combinatorial analysis of bow varieties, as follows.

\smallskip

First, following Witten's arguments \cite{witten}, we provide a combinatorial characterization of quiver varieties within the broader class of bow varieties. Specifically, this characterization requires that the vector $e \in \Z^n$ of D5 local charges be an {\em $m$-bounded non-decreasing sequence of integers}:
\begin{equation*}
e_1 \leq e_2 \leq \dots \leq e_n \leq e_1 + m .
\end{equation*}

Second, we present two combinatorial descriptions of the equivariant K-theory class of the tangent space $T_x\mathcal{M}$, where $x$ is a torus fixed point on the bow variety $\mathcal{M}$. These descriptions are given in Theorems~\ref{thm:sum01pairs} and~\ref{thm:tangentcharm}, which constitute a significant portion of our paper and are among our main results. We believe these theorems will be essential for future geometric studies of instanton moduli spaces. One description uses extended Young diagrams, which generalize the combinatorics typically associated with quiver varieties, while the other uses the combinatorics of Maya diagrams.

Third, we observe that the brane diagram of any bow variety can be transformed into that of a quiver variety through a combinatorial operation we call \emph{D5-swap}. We study the geometric implications of the D5-swap on bow varieties. From the physics perspective, a D5-swap is often (but not always) seen as an innocent transformation that does not alter the underlying theory. Mathematically, a D5-swap can be either a benign change (such as a homotopy equivalence) or a more substantial one. For instance, the fixed point count remains unchanged under D5-swaps (Fig.~\ref{fig:D5swap}), and our Theorem~\ref{thm:D5swap_tangent} explains that a D5-swap resembles a situation in which one variety is a fibration over the other. However, sharper statements are not generally expected, as demonstrated by counterexamples in \cite{ji2023bow}.

Fourth, for quiver varieties, the Bia\l ynicki-Birula cell decomposition induced by a natural one-parameter subgroup covers the entire variety. We provide a necessary and, conjecturally, also sufficient condition for the analogous cell decomposition to cover a bow variety. The condition is that the local charge vector $e$ must be an {\em $m+1$ bounded non-decreasing sequence of integers}, a relaxation of the condition that would otherwise classify the bow variety as a quiver variety. This identifies an interesting sub-class of {\em quiver-like bow varieties}.

\smallskip

After these preparations we can state and prove a refinement of the Euler characteristic computation to motivic expressions and virtual Poincar\'e polynomials, in certain cases. For example, we obtain 
\begin{theorem}[{Theorem~\ref{thm:poincare_zastava}}, Corollary~\ref{cor:Zminus,m=1}]
\label{thm:main3}
Let $m=1$ and $n$ arbitrary. 
\begin{itemize}
    \item
The Poincaré polynomial of the $(-)$-cells of $\mathcal{M}=\mathcal{M}(d,e,f)$ can be expressed as
\[P^-_t(\mathcal{M})=\sum_{(B_1,\dots,B_n)} \prod_{\beta=1}^n t^{2(n|Y_{\beta}|-\beta l(Y_\beta) + \sum_{1 \leq \alpha < \beta} \binom{|e_{\beta}-e_{\alpha}-1|}{2}} 
\] 
where the summation runs over $n$-tuples of extended Young diagrams $(B_1,\dots,B_n)$ with $d$ fixed, $(Y_1,\dots,Y_n)$ are their classical parts and $l(Y_i)$ is the number of columns of $Y_i$. 
\item For certain integers $A(e), B(e)$ we have
\[
\sum_{d \geq 0} P^-_t(\mathcal{M}(d,e,f))q^d = 
q^{A(e)}
\cdot 
t^{2B(e)}
\cdot \prod_{i=1}^n \prod_{l \geq 1}^{\infty} \frac{1}{1-t^{2(nl-i)}q^l},\]
where $P^-_t$ counts virtual Betti numbers arising from the negative weights for a one-parameter subgroup of the torus.
\end{itemize}
\end{theorem}

Similarly, we derive formulas for the case when $n = 1$ and $m$ is arbitrary (see Proposition~\ref{prop:n1marbgenfn}), and examples when $n,m>1$. Additionally, we prove that no product formula exists for the $(+)$ or $(-)$-Poincar\'e polynomials when both $n > 1$ and $m > 1$ (see Corollary~\ref{cor:noprodnmarb}), thereby resolving a question that was open even for quiver varieties.
If the bow variety is quiver-like, the result is even stronger:

\begin{theorem}[{Theorem~\ref{thm:cbbcelldec_cover}}]
\label{thm:main4}
Assume that $e \in \ZZ^n$ is an $m+1$-bounded non-decreasing sequence, that is 
\begin{equation*}
e_1 \leq e_2 \leq \dots \leq e_n \leq e_1 + m +1.
\end{equation*}
Then the Bia\l ynicki-Birula $(+)$-cells cover $\mathcal{M}(d,e,f)$ for any $d \in \ZZ_{\geq 0}$ and $f \in \ZZ^m$. In particular, Theorem~\ref{thm:main3} gives the Poincar\'e polynomial of a retract of $\mathcal{M}(d,e,f)$.
\end{theorem}

Furthermore, we establish results concerning the equivariant homology groups in the special case of $m = 1$, known also as \emph{Zastavas}. We relate the bow variety analogs of Nekrasov partition functions to those of the quiver variety $M(r,n)$---the moduli space parametrizing rank $r$ torsion-free framed sheaves on the projective plane with second Chern number $n$.

\begin{corollary}[{Corollary~\ref{cor:partfunasquivfn}}] 
\label{cor:main5}
Let $m=1$. Assume that $e \in \ZZ^n$ is a $2$-bounded non-decreasing sequence. Then for any $d \in \mathbb{Z}_{\geq 0}$,
\[Z(\varepsilon_1,\varepsilon_2,a,q )= Z^{M(r,n)}(\varepsilon_1,\varepsilon_2,a + \varepsilon_1 e,q) \cdot \prod_{\substack{1 \leq \alpha, \beta \leq n \\ (s_1,s_2) \in R^{e_{\alpha},e_{\beta}}_{\alpha,\beta}}} \frac{1}{ s_1\varepsilon_1+s_2\varepsilon_2+a_{\beta}-a_{\alpha}}.\]
\end{corollary}

The rest of the paper is organised as follows. In Section~\ref{sec:branediags} we recall the definition of brane diagrams and the construction of bow varieties in affine type $\hat{A}$.  Section~\ref{sec:torusaction} explores a torus action on bow varieties and its fixed points. We provide a combinatorial condition for a bow variety to be a quiver variety and investigate the effect of swapping D5 branes in Section~\ref{sec:quivers_D5swaps}. We introduce a core-quotient decomposition for fixed point count, obtain results on generating series of Euler numbers, and prove Theorem~\ref{thm:main1} in Section~\ref{sec:gen_series}. Section~\ref{sec:equivK} discusses the equivariant K-theory of the tangent spaces at the torus fixed points and establishes the crucial Theorems~\ref{thm:sum01pairs} and~\ref{thm:tangentcharm}. We analyze the BB cells and prove Theorems~\ref{thm:main3} and \ref{thm:main4} in Section~\ref{sec:refinedgenseries}. In Section~\ref{sec:partition_function} we introduce the parabolic partition function associated with bow varieties and deduce Corollary~\ref{cor:main5}. Appendix~\ref{sec:modularity} shows that one of our fixed point count generating series is modular, and it also reveals connections to concepts of number theory (the sum-of-divisors function, theta functions).

\subsection*{Acknowledgement} 
The authors are thankful to T. Botta and B. Szendr{\H o}i for helpful discussions on the topic. Á.Gy.~was supported by the János Bolyai Research Scholarship of the Hungarian Academy of Sciences. R.R. was supported by NSF grants 2152309 and 2200867. G.H. was supported by the MTA–HUN-REN RI Lend\"ulet ``Momentum'' Automorphic Research Group and NKFIH (National Research, Development and Innovation Office) grant K~143876.

\section{Brane diagrams and bow varieties of type $\hat{A}$}
\label{sec:branediags}

In this section we recall the combinatorics of brane diagrams and the associated Cherkis bow varieties of type~$\hat{A}$, as defined in \cite[Section~2.2]{nakajima2017cherkis}. We follow the notation of \cite[Section~3.1]{rimanyi2020bow}, but add another $\C^*$ action, and make a few minor notation changes. Familiarity with these two papers may be helpful for the reader. 

\subsection{Brane diagrams, Hanany-Witten transition}
\label{sec:branediagrams}

A brane diagram of type $\hat{A}$ (or affine type $A$) is a circle decorated with some line segments which are slanted either from left to right (these called NS5 branes: $F_1, F_2,\ldots, F_m$), or from right to left (these are called D5 branes: $E_1, E_2, \ldots, E_n$). The 5-branes partition the circle into segments called D3 branes. Each D3 brane is decorated with an integer $d_s$, its \emph{multiplicity} or \emph{dimension}. (Only non-negative multiplicities will be relevant geometrically, but it is combinatorially convenient to permit integers.)

In our figures we will put the 5-branes on the `top arc' of the circle; this way left and right has a meaning. An upper $+$ (respectively $-$) index of a 5-brane will refer to the D3 brane on its right, respectively left, see
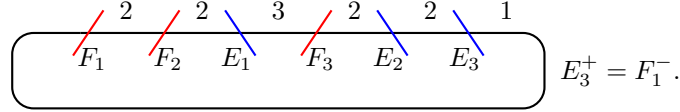
\begin{center}
\begin{tikzpicture}
\draw[thick] (-0.1,0) -- (5.1,0);
\draw[red,thick] (-0.2,-0.3) -- (0.2,0.3);
\draw[red,thick] (0.8,-0.3) -- (1.2,0.3);
\draw[blue,thick] (1.8,0.3) -- (2.2,-0.3);
\draw[red,thick] (2.8,-0.3) -- (3.2,0.3);
\draw[blue,thick] (3.8,0.3) -- (4.2,-0.3);
\draw[blue,thick] (4.8,0.3) -- (5.2,-0.3);
\draw[rounded corners=8pt,thick] (5.1,0) -- (6,0) -- (6,-1) -- (-1,-1) -- (-1,0) -- (-0.1,0);
\draw (0.5,0.3) node {2};
\draw (1.5,0.3) node {2};
\draw (2.5,0.3) node {3};
\draw (3.5,0.3) node {2};
\draw (4.5,0.3) node {2};
\draw (5.5,0.3) node {1};
\draw (0.05,-.35) node {$F_1$};
\draw (1.05,-.35) node {$F_2$};
\draw (1.95,-.35) node {$E_1$};
\draw (3.05,-.35) node {$F_3$};
\draw (3.95,-.35) node {$E_2$};
\draw (4.95,-.35) node {$E_3$};
\draw (7,-.5) node {$E_3^+=F_1^-$.};
\end{tikzpicture}
\end{center}

\begin{definition}
Assume $d_2+d'_2 = d_1+d_3+1$. Carrying out the local change
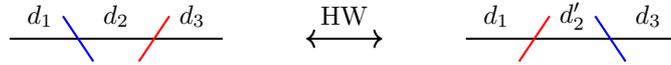
\begin{center}
\begin{tikzpicture}
\draw[thick] (0.1,0) -- (2.9,0);
\draw[blue,thick] (0.8,0.3) -- (1.2,-0.3);
\draw[red,thick] (1.8,-0.3) -- (2.2,0.3);
\draw (0.5,0.3) node {$d_1$};
\draw (1.5,0.3) node {$d_2$};
\draw (2.5,0.3) node {$d_3$};
\draw[<->, thick] (4,0) -- (5,0);
\draw (4.5,0.3) node {HW};
\draw[thick] (6.1,0)--(8.9,0);
\draw[red,thick] (6.8,-0.3) -- (7.2,0.3);
\draw[blue,thick] (7.8,0.3) -- (8.2,-0.3);
\draw (6.5,0.3) node {$d_1$};
\draw (7.5,0.3) node {$d'_2$};
\draw (8.5,0.3) node {$d_3$};
\end{tikzpicture}
\end{center}
(in either direction) in a brane diagram is called a Hanany-Witten transition.
\end{definition}

Both of the diagrams
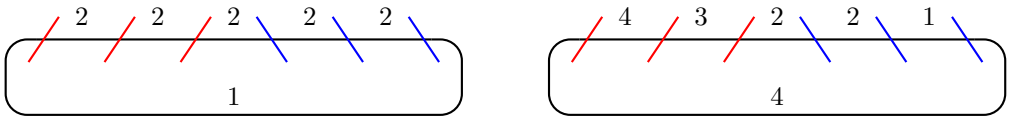
\begin{center}
\begin{tikzpicture}
\draw[thick] (-0.1,0) -- (5.1,0);
\draw[red,thick] (-0.2,-0.3) -- (0.2,0.3);
\draw[red,thick] (0.8,-0.3) -- (1.2,0.3);
\draw[red,thick] (2.2,0.3) -- (1.8,-0.3);
\draw[blue,thick] (3.2,-0.3) -- (2.8,0.3);
\draw[blue,thick] (3.8,0.3) -- (4.2,-0.3);
\draw[blue,thick] (4.8,0.3) -- (5.2,-0.3);
\draw[rounded corners=8pt,thick] (5.1,0) -- (5.5,0) -- (5.5,-1) -- (-0.5,-1) -- (-0.5,0) -- (-0.1,0);
\draw (0.5,0.3) node {2};
\draw (1.5,0.3) node {2};
\draw (2.5,0.3) node {2};
\draw (3.5,0.3) node {2};
\draw (4.5,0.3) node {2};
\draw (2.5,-0.75) node {1};
\end{tikzpicture}
\hskip 1 true cm
\begin{tikzpicture}
\draw[thick] (-0.1,0) -- (5.1,0);
\draw[red,thick] (-0.2,-0.3) -- (0.2,0.3);
\draw[red,thick] (0.8,-0.3) -- (1.2,0.3);
\draw[red,thick] (2.2,0.3) -- (1.8,-0.3);
\draw[blue,thick] (3.2,-0.3) -- (2.8,0.3);
\draw[blue,thick] (3.8,0.3) -- (4.2,-0.3);
\draw[blue,thick] (4.8,0.3) -- (5.2,-0.3);
\draw[rounded corners=8pt,thick] (5.1,0) -- (5.5,0) -- (5.5,-1) -- (-0.5,-1) -- (-0.5,0) -- (-0.1,0);
\draw (0.5,0.3) node {4};
\draw (1.5,0.3) node {3};
\draw (2.5,0.3) node {2};
\draw (3.5,0.3) node {2};
\draw (4.5,0.3) node {1};
\draw (2.5,-0.75) node {4};
\end{tikzpicture}
\end{center}
 are Hanany-Witten equivalent to the one above.

\subsection{Bow varieties}
\label{subsec:bowvarieties}

To each D3 brane $X$, we assign a vector space $V_X$ of dimension $d_X$. In addition, we assign a one dimensional vector space $\mathbb{C}_E$ to each D5 brane $E$.

Consider the following vector spaces together with their $\C^*_{t_1} \times \C^*_{t_2}$-action. 

\begin{itemize}
\item For a D5 brane $E$ take
\[ 
\begin{multlined}
\mathbb{M}_E = \mathrm{Hom}(V_{E^+},V_{E^-}) \oplus t_1t_2 \mathrm{Hom}(V_{E^+},\CC_E)\oplus \mathrm{Hom}(\CC_E,V_{E^-}) \\ \oplus t_1t_2 \mathrm{End}(V_{E^-}) \oplus t_1t_2 \mathrm{End}(V_{E^+})
\end{multlined}
\]
with elements $(A_E, a_E, b_E, B_{E}, B'_{E})$.

\item For an NS5 brane $F$ let 
\[ 
\mathbb{M}_Z = t_1\mathrm{Hom}(V_{F^+},V_{F^-}) \oplus t_2\mathrm{Hom}(V_{F^-},V_{F^+}) 
\]
with elements $(C_F, D_F)$.
\item For a D5 brane $E$ let
$ \mathbb{N}_E = t_1t_2\mathrm{Hom}(V_{E^+},V_{E^-})$.
\item For a D3 brane $X$ let $\mathbb{N}_X = t_1t_2\,\mathrm{End}(V_X)$.
\end{itemize}
Consider the sums
\[ \mathbb{M}=\bigoplus_{E\; \mathrm{D5}} \mathbb{M}_E \oplus \bigoplus_{F\; \mathrm{NS5}} \mathbb{M}_F,
\qquad\qquad
\mathbb{N}=\bigoplus_{E\; \mathrm{D5}} \mathbb{N}_E \oplus \bigoplus_{X\; \mathrm{D3}} \mathbb{N}_X
\]
and define the moment map $\mu: \mathbb{M} \to \mathbb{N}$ as follows.
\begin{itemize}
\item The $\mathbb{N}_E$ component for each D5 brane $E$ of $\mu$ is 
\[ B_{E^-}A_E-A_EB'_{E^+}+a_Eb_E.\]
\item The $\mathbb{N}_X$ component of $\mu$ for each D3 brane $X$ depends on the slopes of the 5-branes on the boundaries of $X$ as follows:
\[
\begin{cases}
B_{X^{-}}-B_{X^{+}} & \quad \textrm{for \textbackslash -\textbackslash  } \\
C_{X^{+}}D_{X^{+}}-D_{X^{-}}C_{X^{-}} & \quad \textrm{for  /-/} \\
-D_{X^{-}}C_{X^{-}}-B_{X^{+}} & \quad \textrm{for  /-\textbackslash}  \\
C_{X^{+}}D_{X^{+}}+B_{X^{-}} & \quad \textrm{for  \textbackslash -/}.
\end{cases}
\]
\end{itemize}
Let $\widetilde{\mathcal{M}}$ be the points of $\mu^{-1}(0)$ for which the stability conditions (S1), (S2) hold.
\begin{enumerate}
\item[(S1)] There is no nonzero subspace $0 \neq S \subset V_{E^{+}}$ with $B_E (S) \subset S$, $A_E(S)=0 = b_E(S)$.
\item[(S2)] There is no proper subspace $S \subsetneq V_{E^{-}}$ with $B_E (T) \subset T$, $\mathrm{Im}(A_E)+\mathrm{Im}(a_E) \subset T$.
\end{enumerate}
Let $\mathcal{G}=\prod_{E \; \mathrm{D3}} GL(V_X)$ act on $\widetilde{\mathcal{M}} \times \mathbb{C}$ by $g.(m,z)=(gm,\chi(g)^{-1}z)$ utilizing the character
\[ \chi: \mathcal{G} \to \mathbb{C}^{\ast}, \quad (g_{X})_X \mapsto \prod_{X}\det(g_X).\]
A point $m \in \widetilde{\mathcal{M}}$ is called \emph{stable} if the orbit $\mathcal{G}.(m,z)$ is closed and the stabilizer of $(m,z)$ is finite for $z \neq 0$. The bow variety associated with the above data is defined as the quotient \[\mathcal{M} \coloneqq\widetilde{\mathcal{M}}^s/\mathcal{G}\] of the set of stable points of $\widetilde{\mathcal{M}}$ with $\mathcal{G}$.

\begin{theorem}[{\cite{nakajima2017cherkis}}] For a generic choice of stability condition the variety $\mathcal{M}$ is smooth and has a holomorphic symplectic structure. Bow varieties associated with Hanany-Witten equivalent diagrams are isomorphic. 
\end{theorem}

In this paper we consider bow varieties almost exclusively with generic stability condition.

\begin{example}
    For $n=m=1$ the associated bow variety (if not empty) is the Hilbert scheme of certain number of points on the plane. For example for $d_0=11, d_1=14$, or for $d_0=d_1=8$ the bow variety is $(\C^2)^{[8]}$, cf. Example \ref{ex:hilb2}
\end{example}

\begin{example}
    The bow variety of the example in Section \ref{sec:branediagrams} is a smooth  holomorphic symplectic variety of dimension 10. It has a $(\C^*)^3\times \C^*_{t_1}\times \C^*_{t_2}$ torus action (see Section \ref{sec:torusaction} below) with 13 fixed points. 
\end{example}

\subsection{Brane diagrams in standard form}
\label{subsec:brane_diags_standard_form}
By a sequence of Hanany-Witten transitions we may rearrange the brane diagram to the following {\em standard form}: on the top arc of the circle the NS5 branes are on the left, and the D5 branes are on the right. 
For such a brane diagram let us call the 5-branes, from left to right
\[
F_1, F_{2}, \ldots, F_{m-1}, F_m, E_1, E_2, \ldots, E_{n-1}, E_n.
\]
and define the following integers (`local brane charges'):
\[
e_i=d_{E_i^+} - d_{E_i^-}, \qquad\qquad
f_j=d_{F_j^-} - d_{F_j^+},
\]
and denote the multiplicity of the D3 brane in between $F_m$ and $E_1$ by $d_0$, or simply by $d$.

Clearly we have $\sum_i e_i=\sum_j f_j$, and the vectors $e,f$ together with the number $d$ determine all multiplicities. We will call this brane diagram ``$(d,e,f)$ standard diagram'', and the associated bow variety we denote by $\mathcal{M}(d,e,f)$.

\medskip
\noindent We introduce two types of moves on triples $(d,e,f)$:

[move-1] 
    $
    (d,e,f)\mapsto (
    (d+e_1,(e_2, e_3, \ldots, e_n, e_1+m),
    (f_1+1, f_2+1, \ldots, f_m+1)),
    $

[move-2] 
    $
    (d,e,f)\mapsto (d+f_m,
    (e_1+1, e_2+1, \ldots, e_n+1),
    (f_m+n,f_1, f_2, \ldots, f_{m-1})).
    $

\begin{lemma} Any fixed Hanany-Witten equivalence class of brane diagrams contains infinitely many standard diagrams. The corresponding triples $(d,e,f)$ can be obtained from each other by a sequence of [move-1], [move-2] and their inverses.
\end{lemma}
\begin{proof} Take $E_1$ and move it through consecutively the NS5 branes with HW transformations. Eventually, it will arrive to the left of $F_1$, so it can be moved on the bottom arc to the right of $E_n$. This yields [move-1]. When moving $F_m$ consecutively through the D5 branes we get [move-2].
\end{proof}

\begin{example}
    The two brane diagrams at the end of Section \ref{sec:branediagrams} are in standard form, and they are Hanany-Witten equivalent. 
\end{example}

\begin{example} \label{ex:hilb2}
    The $m=n=1$ case is already interesting. Using the moves above (that is, Hanany-Witten transitions), the general standard form $(d, (k), (k))$ can be brought to the form $(d',(0),(0))$. The bow variety associated to the latter is $(\C^2)^{[d']}$, the Hilbert scheme of $d'$ points on $\C^2$. Easy calculation gives $d'=d-k(k-1)/2$.
\end{example}

\section{Torus action}
\label{sec:torusaction}

In Section \ref{subsec:bowvarieties} we indicated the action of 
the torus $(\C^*)^2=\C^*_{t_1} \times \C^*_{t_2}$  on $\mathbb M$ and $\mathbb N$.
It is straightforward to verify that the map $\mu$ is equivariant, and hence we obtain a $\C^*_{t_1} \times \C^*_{t_2}$-action on the bow variety. 
Moreover, the group $(\C^*)^n$ of linear reparametrizations of the lines $\C_E$ also acts on the bow variety, and in turn, we have an action of 
\[
T=(\C^*_{t_1} \times \C^*_{t_2}) \times (\C^*)^{\text{D5 branes}}
=(\C^*)^{2} \times (\C^*)^{n}
=(\C^*)^{n+2}.
\]

\subsection{Tangent space}
Consider the three term complex
\[ \bigoplus_{X\; \mathrm{D3}} \mathrm{End}(V_X) \xrightarrow{\sigma}\mathbb{M}\xrightarrow{\tau=\begin{pmatrix} \tau_1 \\ \tau_2 \end{pmatrix}} \mathbb{N}, \]
\noindent where
\[ \sigma:\bigoplus_{X\; \mathrm{D3}} v_X \mapsto \begin{pmatrix} [v_{X^{\pm}},B_{X^{\pm}}] \\ v_{X^{+}}A-Av_{X^{-}} \\ v_{X^{+}}a \\ -bv_{X^{-}} \\ v_{F^{+}}C-Cv_{F^{-}} \\ v_{F^{-}}D-Dv_{F^{+}} \end{pmatrix}, \]
\[
\tau_1: \begin{pmatrix} \overline{B} \\ \overline{A} \\ \overline{a} \\ \overline{b} \\ \overline{C} \\ \overline{D} \end{pmatrix} \mapsto B_{{X^+}}\overline{A}-\overline{A}B_{{X^-}}+\overline{B}_{{X^+}}A-A\overline{B}_{{X^-}}+\overline{a}b+a\overline{b},
\]
\noindent
and $\tau_2$ is
\[ -\overline{B}_{{X^+}}+\overline{B}'_{{X^-}}, \quad \overline{C}_{X^+}D_{X^+}+C_{X^+}\overline{D}_{X^+}-\overline{D}_{X^-}C_{X^-}-D_{X^-}\overline{C}_{X^-},\]
\[\overline{C}_{X^+}D_{X^+}+C_{X^+}\overline{D}_{X^+}+\overline{B}_{{X^-}},\quad -\overline{B}_{{X^+}}-\overline{D}_{X^-}C_{X^-}-D_{X^-}\overline{C}_{X^-}.\]
\begin{lemma}
\label{lem:complex_is_equivariant}
\begin{enumerate}
\item The above complex is $T$ equivariant.
\item The tangent space of $\mathcal{M}$ at a fixed point represented by $(A,B,a,b,C,D)$ is $\mathrm{Ker} \,\tau / \mathrm{Im}\, \sigma$.
\end{enumerate}
\end{lemma}
\begin{proof}
Part (1) is a straightforward verification. Part (2) is \cite[Section~2.5]{nakajima2017cherkis}.
\end{proof}

\subsection{Torus fixed points}
\label{subsec:torus_fixed_points}
The action of $T$ has isolated fixed points in $\mathcal{M}$. These fixed points $\mathcal{M}^T$ can be described combinatorially in various ways. One description is via \emph{(generalized) Maya diagrams} \cite[Appendix~A]{nakajima2018towards} (cf. \cite[Appendix]{rimanyi2020bow})  that we recall now.

By \cite[Section~6.9.2]{nakajima2017cherkis} there exists a set of coordinates $s_1,\dots,s_{n}$ on $(\mathbb{C}^{\ast})^n \subset T$ such that the action on $\CC_{i}$ 
is induced by 
$ \u_i \coloneqq s_1\dots s_{i}$ for $1 \leq i \leq n$.
Consider a representative $(A,B,a,b,C,D)$ of a point in $\mathcal{M}^{T}$.
Then there must exist a homomorphism $\rho:T \to \mathcal{G}$ such that
\[ s_1\dots s_{n}A_1=\rho_{E^{-}}(t)^{-1} A_1 \rho_{E^{+}}(t) \]
\[ A_E=\rho_{E^{-}}(t)^{-1} A_E \rho_{E^{+}}(t),\quad E \neq E_1 \]
\[ t_1t_2B_X=\rho_{X}(t)^{-1}B_X\rho_{X}(t) \]
\[ a_E(s_1\dots s_{i})^{-1}=\rho_{E^{-}}(t)^{-1}a_E \]
\[ t_1t_2(s_1\dots s_{n})b_1=b_1\rho_{E_1^{-}}(t) \]
\[ t_1t_2(s_1\dots s_{i})b_E=b_E\rho_{E^{-}}(t),\quad E \neq E_1 \]
\[ t_1C_F=\rho_{F^{-}}(t)^{-1} C_F \rho_{F^{+}}(t) \]
\[ t_2D_F=\rho_{F^{+}}(t)^{-1} D_F \rho_{F^{-}}(t) \]
The space 
\[V=\bigoplus_{X\; \mathrm{D3}} V_X\] decomposes into weight spaces with respect to $\rho$:
\[ V=\bigoplus V(t_1,t_2,s_1,\dots,s_{n}).\]
By the above, $A_E$, $E\neq E_1$ preserves the $(t,s)$-weight, $A_1$ shifts the $s$-weights by $(-1,\dots,-1)$, $B_X$ preserves the $s$-weight and decreases the $t$-weight by $(1,1)$, $a_i$ sends $\mathbb{C}_i$ to the weight space $s_1\dots s_{i}$ (every other weight is 0), and $b_i$ is zero on every weight space other than $t_1t_2(s_1\cdots s_{i})$.

For $1 \leq i \leq n$, set
\[ V^i = \bigoplus V^i(t_1,t_2) \]
where
\[ V^i(t_1,t_2)= \oplus_{k \in \ZZ} V(t_1,t_2,s_1^{1+k},\dots,s_{i}^{1+k},s_{i+1}^{k},\dots,s_{n}^{k} ).\]
Then $V^i$ is a $T^2$-module and 
\[V=\oplus_{1 \leq i \leq n} V^i.\] By the above, $\mathbb{C}_{i}$ can only interact with $V^i$, and hence the six-tuple $(A,B,a,b,C,D)$ also decomposes into a direct sum. 

Let us restrict our attention on a single summand $V^i$. As $a_{j}$, $b_j$ are zero for $j\neq i$, $A_j$ is an isomorphism due to stability conditions (S1)--(S2). Therefore, we can identify $V_{j^{-}}$ with $V_{j^{+}}$, and the datum can be represented by a bow diagram with only one D5 brane. Instead of drawing this diagram on the cycle, we draw it on the universal cover, that is, a periodic brane on the infinte line. The D5 branes (all mapping onto $E_i$ under the exponential map) appear at, let's say, integer positions. The NS5 branes come in groups; these are called \emph{blocks} and their positions are indexed by half-integers. This bow diagram and its block decomposition is compatible with the weight decomposition of the underlying vector spaces, because the map $A_1$ shifts the $s$-weights by $(-1,\dots,-1)$, while $A_i$, $i \neq 1$ preserves the weights.


\subsection{Combinatorial codes of torus fixed points---Maya diagrams}
\label{sec:maya}

We associate a 01 sequence to each D5 brane $E_i$.  
It follows from the stability conditions that (when restricted to $V^i$) the maps $D_j$ that appear in the NS5 branes left (resp. right) of $E_i$ are all injective (resp. surjective) and the dimension of the underlying spaces at consecutive $D3$ branes is either constant or increases (resp. decreases) by 1.
Take the 01 sequence where for negative (resp. positive) blocks 1 (resp. 0) indicates the places where the dimension increases (resp. decreases) and 0 (resp. 1) occurring at the places where the dimension is constant. Although $V^i$ is finite dimensional, extend the sequence obtained this way by 0's on the left and by 1's on the right to get a sequence $M_i$ infinite in both directions.


The {\em (generalized) Maya diagram} of the fixed point is obtained by placing the 01 sequence $M_i$ for each $1 \leq i \leq n$ under each other. These Maya diagrams hence comprise sequences of 01 matrices of size $n \times m$, again numbered with half integers. 


\begin{example}\label{ex:2Maya}
The two Maya diagrams 
\[
\begin{tikzpicture}
\draw (-2.2,-.2) -- (3.4,-0.2);\draw (-2.2,1) -- (3.4,1);
\draw (-.2,-0.5) -- (-.2,1.3);
\draw (0.2,1.4) node {$\frac{1}{2}$};
\draw (1,1.4) node {$\frac{3}{2}$};
\draw (1.8,1.4) node {$\frac{5}{2}$};
\draw (2.6,1.4) node {$\frac{7}{2}$};
\draw (-.6,1.4) node {$\frac{-1}{2}$};
\draw (-1.4,1.4) node {$\frac{-3}{2}$};
 \draw (0,0) node {$1$};
 \draw (.4,0) node {$0$};
 \draw (0,.4) node {$0$};
 \draw (.4,.4) node {$1$};
 \draw (0,.8) node {$1$};
 \draw (.4,.8) node {$0$};
 \draw (-.2,-.2)--(0.6,-0.2) -- (0.6,1) -- (-0.2,1) -- (-.2,-.2);
 \draw (0.8,0) node {$1$};
 \draw (1.2,0) node {$0$};
 \draw (0.8,.4) node {$0$};
 \draw (1.2,.4) node {$1$};
 \draw (0.8,.8) node {$0$};
 \draw (1.2,.8) node {$1$};
 \draw (.6,-0.2)--(1.4,-0.2) -- (1.4,1) -- (0.6,1) -- (0.6,-.2);
 \draw (1.6,0) node {$1$};
 \draw (2,0) node {$1$};
 \draw (1.6,.4) node {$1$};
 \draw (2,.4) node {$0$};
 \draw (1.6,.8) node {$1$};
 \draw (2,.8) node {$1$};
 \draw (1.4,-0.2)--(2.2,-0.2) -- (2.2,1) -- (1.4,1) -- (1.4,-.2);
 \draw (2.4,0) node {$1$};
 \draw (2.8,0) node {$1$};
 \draw (2.4,.4) node {$1$};
 \draw (2.8,.4) node {$1$};
 \draw (2.4,.8) node {$1$};
 \draw (2.8,.8) node {$1$};
 \draw (2.2,-0.2)--(3,-0.2) -- (3,1) -- (2.2,1) -- (2.2,-.2);
 \draw (-.8,0) node {$1$};
 \draw (-.4,0) node {$0$};
 \draw (-0.8,.4) node {$0$};
 \draw (-.4,.4) node {$0$};
 \draw (-0.8,.8) node {$0$};
 \draw (-.4,.8) node {$0$};
 \draw (-1,-.2)--(-0.2,-0.2) -- (-0.2,1) -- (-1,1) -- (-1,-.2);
 \draw (-1.6,0) node {$0$};
 \draw (-1.2,0) node {$0$};
 \draw (-1.6,.4) node {$0$};
 \draw (-1.2,.4) node {$0$};
 \draw (-1.6,.8) node {$0$};
 \draw (-1.2,.8) node {$0$};
 \draw (-1.8,-.2)--(-1,-0.2) -- (-1,1) -- (-1.8,1) -- (-1.8,-.2);
\draw (3.4,.5) node {$\forall 1$};
\draw (-2.2,.5) node {$\forall 0$};
\draw ( 5,0) node {$E_3$};
\draw ( 5,.4) node {$E_2$};
\draw ( 5,.8) node {$E_1$};
\draw (0,-0.6) node {$F_1$};
\draw (0.4,-0.6) node {$F_2$};
\draw (0.8,-0.6) node {$F_1$};
\draw (1.2,-0.6) node {$F_2$};
\draw (-.8,-0.6) node {$F_1$};
\draw (-0.4,-0.6) node {$F_2$};
\draw (1.8,-0.6) node {$\cdots$};
\draw (-1.4,-0.6) node {$\cdots$};
\end{tikzpicture}
\]
\[
\begin{tikzpicture}
\draw (-2.2,-.2) -- (3.4,-0.2);\draw (-2.2,1) -- (3.4,1);
\draw (-.2,-0.5) -- (-.2,1.3);
\draw (0.2,1.4) node {$\frac{1}{2}$};
\draw (1,1.4) node {$\frac{3}{2}$};
\draw (1.8,1.4) node {$\frac{5}{2}$};
\draw (2.6,1.4) node {$\frac{7}{2}$};
\draw (-.6,1.4) node {$\frac{-1}{2}$};
\draw (-1.4,1.4) node {$\frac{-3}{2}$};
 \draw (0,0) node {$0$};
 \draw (.4,0) node {$0$};
 \draw (0,.4) node {$0$};
 \draw (.4,.4) node {$0$};
 \draw (0,.8) node {$1$};
 \draw (.4,.8) node {$0$};
 \draw (-.2,-.2)--(0.6,-0.2) -- (0.6,1) -- (-0.2,1) -- (-.2,-.2);
 \draw (0.8,0) node {$1$};
 \draw (1.2,0) node {$1$};
 \draw (0.8,.4) node {$1$};
 \draw (1.2,.4) node {$0$};
 \draw (0.8,.8) node {$0$};
 \draw (1.2,.8) node {$0$};
 \draw (.6,-0.2)--(1.4,-0.2) -- (1.4,1) -- (0.6,1) -- (0.6,-.2);
 \draw (1.6,0) node {$1$};
 \draw (2,0) node {$1$};
 \draw (1.6,.4) node {$1$};
 \draw (2,.4) node {$1$};
 \draw (1.6,.8) node {$1$};
 \draw (2,.8) node {$1$};
 \draw (1.4,-0.2)--(2.2,-0.2) -- (2.2,1) -- (1.4,1) -- (1.4,-.2);
 \draw (2.4,0) node {$1$};
 \draw (2.8,0) node {$1$};
 \draw (2.4,.4) node {$1$};
 \draw (2.8,.4) node {$1$};
 \draw (2.4,.8) node {$1$};
 \draw (2.8,.8) node {$1$};
 \draw (2.2,-0.2)--(3,-0.2) -- (3,1) -- (2.2,1) -- (2.2,-.2);
 \draw (-.8,0) node {$0$};
 \draw (-.4,0) node {$1$};
 \draw (-0.8,.4) node {$0$};
 \draw (-.4,.4) node {$0$};
 \draw (-0.8,.8) node {$0$};
 \draw (-.4,.8) node {$0$};
 \draw (-1,-.2)--(-0.2,-0.2) -- (-0.2,1) -- (-1,1) -- (-1,-.2);
 \draw (-1.6,0) node {$0$};
 \draw (-1.2,0) node {$0$};
 \draw (-1.6,.4) node {$0$};
 \draw (-1.2,.4) node {$0$};
 \draw (-1.6,.8) node {$1$};
 \draw (-1.2,.8) node {$0$};
 \draw (-1.8,-.2)--(-1,-0.2) -- (-1,1) -- (-1.8,1) -- (-1.8,-.2);
\draw (3.4,.5) node {$\forall 1$};
\draw (-2.2,.5) node {$\forall 0$};
\draw ( 5,0) node {$E_3$};
\draw ( 5,.4) node {$E_2$};
\draw ( 5,.8) node {$E_1$};
\end{tikzpicture}
\]
encode two of the 1806 torus fixed points on the 36-dimensional bow variety 
with $d=6$, $e=(2,3,1)$, $f=(2,4)$.
\end{example}

\medskip

The collection of Maya diagram for  $\mathcal{M}(d,e,f)$ is the ones satisfying the following properties (see \cite[Appendix~A]{rimanyi2020bow}).
\begin{enumerate}
    \item[(i)] For large enough $k$ all entries of block $\frac{k}{2}$ are 1, and all entries of block $\frac{-k}{2}$ are 0.
    \item [(ii)] For $1 \leq i \leq n$ the entries satisfy
    \[ e_i = 
    \#\{ \textrm{0s in row }i\textrm{ of positive blocks} \}
    -
    \#\{ \textrm{1s in row }i\textrm{ of negative blocks}\}
    . \]
    \item[(iii)] For $1 \leq j \leq m$ the entries satisfy
    \begin{multline*} 
    \ \hskip 0.8 true cm
    f_j = 
    \#\{ \textrm{0s in column }F_j\textrm{ of positive blocks} \}
    \\ -
    \#\{ \textrm{1s in column }F_j\textrm{ of negative blocks} \}. \hskip 0.8 true cm \ 
    \end{multline*}
    \item[(iv)]
    \[
    \begin{aligned}
    d = &\#\{ \textrm{1s in block }\frac{-1}{2}\}+2\#\{ \textrm{1s in block }\frac{-3}{2}\}+3\#\{ \textrm{1s in block }\frac{-5}{2}\}+\dots \\
    & + \#\{ \textrm{0s in block }\frac{3}{2}\}+2\#\{ \textrm{0s in block }\frac{5}{2}\}+3\#\{ \textrm{0s in block }\frac{7}{2}\}+\dots.
    \end{aligned} 
    \]
\end{enumerate}

For a pair of integer vectors $(e,f)\in \ZZ^n \times \ZZ^m$ and a nonnegative integer $d\in \ZZ_{\geq 0}^+$ let $M(d,e,f)$ be the set of Maya diagrams associated with fixed points on the bow variety $\mathcal{M}(d,e,f)$. Define
\[ 
M(e,f) = \bigsqcup_{d \geq 0} M(d,e,f).
\]

\subsection{Combinatorial codes of torus fixed points---tie diagrams}
There is another combinatorial code for torus fixed points, besides Maya diagrams: tie diagrams. They are closer to the physical origin of brane configurations. In tie diagrams D3 branes are realized by a number of \emph{ties}, on each interval between 5-branes as many ties as the D3 multiplicity is there. A tie connects an NS5 brane to a D5 brane and goes along the circle of our diagram certain number of times. A tie diagram can have 0 or 1 tie for each triple (NS5 brane, D5 brane, winding number). In Figure \ref{fig:tiediagrams} we illustrate a tie diagram of a torus fixed point, whose Maya diagram is
\[
\begin{tikzpicture}
\draw (-2.2,-.2) -- (1.8,-0.2);
\draw (-2.2,1) -- (1.8,1);
\draw (-.2,-0.5) -- (-.2,1.3);
\draw (0.2,1.4) node {$\frac{1}{2}$};
\draw (1,1.4) node {$\frac{3}{2}$};
\draw (-.6,1.4) node {$\frac{-1}{2}$};
\draw (-1.4,1.4) node {$\frac{-3}{2}$};
 \draw (0,0) node {$1$};
 \draw (.4,0) node {$0$};
 \draw (0,.4) node {$1$};
 \draw (.4,.4) node {$1$};
 \draw (0,.8) node {$1$};
 \draw (.4,.8) node {$1$};
 \draw (-.2,-.2)--(0.6,-0.2) -- (0.6,1) -- (-0.2,1) -- (-.2,-.2);
 \draw (0.8,0) node {$1$};
 \draw (1.2,0) node {$0$};
 \draw (0.8,.4) node {$1$};
 \draw (1.2,.4) node {$1$};
 \draw (0.8,.8) node {$1$};
 \draw (1.2,.8) node {$1$};
 \draw (.6,-0.2)--(1.4,-0.2) -- (1.4,1) -- (0.6,1) -- (0.6,-.2);
 \draw (-.8,0) node {$0$};
 \draw (-.4,0) node {$0$};
 \draw (-0.8,.4) node {$1$};
 \draw (-.4,.4) node {$0$};
 \draw (-0.8,.8) node {$0$};
 \draw (-.4,.8) node {$1$};
 \draw (-1,-.2)--(-0.2,-0.2) -- (-0.2,1) -- (-1,1) -- (-1,-.2);
 \draw (-1.6,0) node {$0$};
 \draw (-1.2,0) node {$0$};
 \draw (-1.6,.4) node {$0$};
 \draw (-1.2,.4) node {$1$};
 \draw (-1.6,.8) node {$0$};
 \draw (-1.2,.8) node {$0$};
 \draw (-1.8,-.2)--(-1,-0.2) -- (-1,1) -- (-1.8,1) -- (-1.8,-.2);
\draw (1.8,.5) node {$\forall 1$};
\draw (-2.2,.5) node {$\forall 0$};
\draw ( 3,0) node {$E_3$};
\draw ( 3,.4) node {$E_2$};
\draw ( 3,.8) node {$E_1$};
\draw (0,-0.6) node {$F_1$};
\draw (0.4,-0.6) node {$F_2$};
\end{tikzpicture}
\]

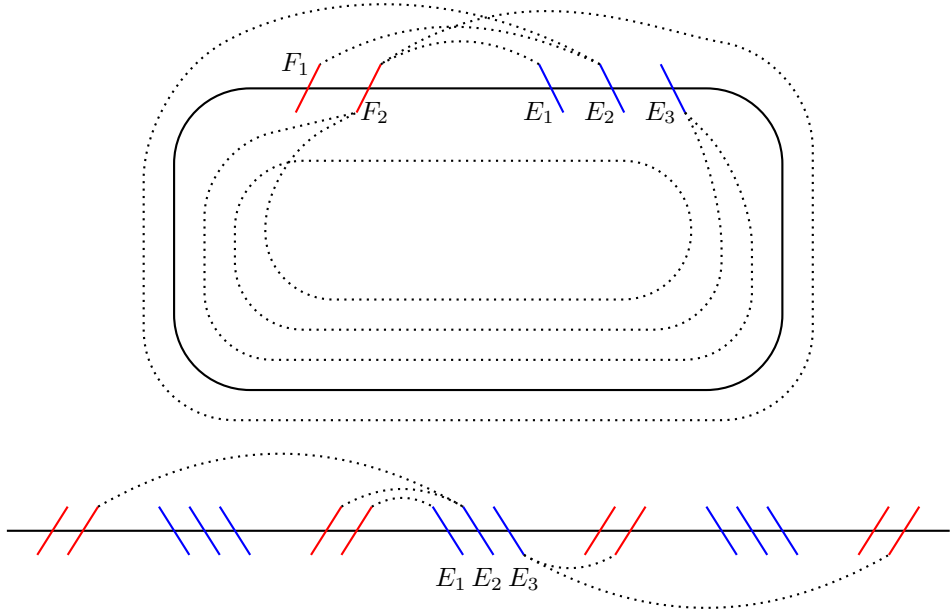
\begin{figure}
\[
\begin{tikzpicture}[scale=0.8]
    \draw[rounded corners=10mm, thick] (0,5) -- (10,5) -- (10,10) --(0,10) -- cycle;
    \draw[thick,red] (2,9.6)--(2.4,10.4);
    \draw[thick,red] (3,9.6)--(3.4,10.4);
    \draw[thick,blue] (6.4,9.6)--(6,10.4);
    \draw[thick,blue] (7.4,9.6)--(7,10.4);
    \draw[thick,blue] (8.4,9.6)--(8,10.4);
    
\draw[thick,dotted] (6,10.4) to [out=150,in=30] (3.4,10.4);
\draw[thick,dotted] (7,10.4) to [out=155,in=30] (2.4,10.4);

\draw[thick, dotted,rounded corners=12mm] (7,10.4) to [out=150,in=30] (-.5,10) -- (-.5,4.5)-- (10.5,4.5) -- (10.5,10.5) to [out=170,in=35] (3.4,10.4);

\draw[thick, dotted,rounded corners=9mm] (8.4,9.6) -- (9.5,8.5) -- (9.5,5.5) -- (0.5,5.5) -- (0.5,9) -- (3,9.6);
\draw[thick, dotted,rounded corners=9mm] (8.4,9.6) -- (9,8.5) -- (9,6) -- (1,6) -- (1,8.8) -- (8.5,8.8) -- (8.5,6.5) -- (1.5,6.5)  -- (1.5,8.8) --  (3,9.6);

\draw (2,10.4) node {$F_1$};
\draw (3.3,9.6) node {$F_2$};
\draw (6,9.6) node {$E_1$};
\draw (7,9.6) node {$E_2$};
\draw (8,9.6) node {$E_3$};
\end{tikzpicture}
\]
\[
\begin{tikzpicture}[scale=.8]
\draw[thick]  (0.5,0) -- (16,0);
        \draw[thick,red] (1,-.4)--(1.5,.4);
        \draw[thick,red] (1.5,-.4)--(2,.4);
    \draw[thick,blue] (3.5,-.4)--(3,.4);
    \draw[thick,blue] (4,-.4)--(3.5,.4);
    \draw[thick,blue] (4.5,-.4)--(4,.4);
        \draw[thick,red] (5.5,-.4)--(6,.4);
        \draw[thick,red] (6,-.4)--(6.5,.4);
    \draw[thick,blue] (8,-.4)--(7.5,.4);
    \draw[thick,blue] (8.5,-.4)--(8,.4);
    \draw[thick,blue] (9,-.4)--(8.5,.4);
        \draw[thick,red] (10,-.4)--(10.5,.4);
        \draw[thick,red] (10.5,-.4)--(11,.4);
    \draw[thick,blue] (12.5,-.4)--(12,.4);
    \draw[thick,blue] (13,-.4)--(12.5,.4);
    \draw[thick,blue] (13.5,-.4)--(13,.4);
        \draw[thick,red] (14.5,-.4)--(15,.4);
        \draw[thick,red] (15,-.4)--(15.5,.4);
\draw[thick,dotted] (2,0.4) to [out=30,in=150] (8,0.4);
\draw[thick,dotted] (6,0.4) to [out=30,in=150] (8,0.4);
\draw[thick,dotted] (6.5,0.4) to [out=30,in=150] (7.5,0.4);
\draw[thick,dotted] (9,-0.4) to [out=-30,in=-150] (10.5,-0.4);
\draw[thick,dotted] (9,-0.4) to [out=-30,in=-150] (15,-0.4);
\draw (7.8,-.8) node {$E_1$};
\draw (8.4,-.8) node {$E_2$};
\draw (9,-.8) node {$E_3$};
\end{tikzpicture}
\]
\caption{The tie diagram, and its universal cover, of one of the 41 fixed points of $\mathcal{M}(5,(-1,-2,2),(-1,0))$.}\label{fig:tiediagrams}
\end{figure}
The correspondence between tie diagrams and Maya diagrams is as follows. Consider the universal cover of the tie diagram with the choice that all ties are lifted to a representative corresponding to a distinguished `block' of D5 branes (see Figure \ref{fig:tiediagrams}). Then the possible ties from this group to groups of lifted NS5 branes come in blocks. These blocks correspond to the blocks of the Maya diagrams. For positive blocks (ties to the right) the rule is: {\em the $(E,F)$ entry is 0 if there is a $E$-$F$ tie, otherwise it is 1}; and for negative blocks (ties to the left) the rule is: {\em the $(E,F)$ entry is 1 if there is a $E$-$F$ tie, otherwise it is 0.} For more details see \cite[Appendix]{rimanyi2020bow}.

\section{Quivers, D5 swaps}
\label{sec:quivers_D5swaps}

\subsection{Co-balanced brane diagrams} 
\label{subsec:balanced_brane_diagrams}
In this subsection we follow the arguments of \cite[Sect. 3.3]{witten}.
If $d_{E^-}=d_{E^+}$ for all D5 branes, we call the diagram \emph{co-balanced}. A co-balanced brane diagram is the cyclic concatenation of pieces like the picture on the left.
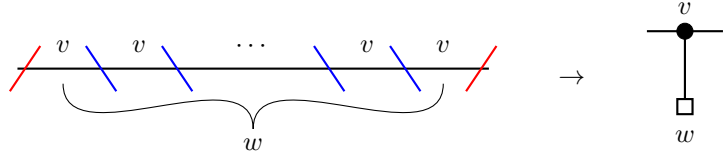
\[
\begin{tikzpicture}[baseline=-.2cm]
\draw[thick] (-0.1,0) -- (6.1,0);
\draw[red,thick] (-0.2,-0.3) -- (0.2,0.3);
\draw[blue,thick] (1.2,-0.3) -- (0.8,0.3);
\draw[blue,thick] (1.8,0.3) -- (2.2,-0.3);
\draw[blue,thick] (3.8,0.3) -- (4.2,-0.3);
\draw[blue,thick] (4.8,0.3) -- (5.2,-0.3);
\draw[red,thick] (6.2,0.3) -- (5.8,-0.3);
\draw (0.5,0.3) node {$v$};
\draw (1.5,0.3) node {$v$};
\draw (3,0.3) node {$\cdots$};
\draw (4.5,0.3) node {$v$};
\draw (5.5,0.3) node {$v$};
\draw (0.5,-.2 ) to[out=-90,in=90] (3,-.8) to[out=90,in=-90] (5.5,-.2);
\draw (3,-1) node {$w$};
\end{tikzpicture}
\qquad\rightarrow\qquad
\begin{tikzpicture}[baseline=.3cm]
\draw[,thick] (-.5,1) -- (.5,1);
\draw[fill] (0,1) circle (3pt);     
\node at (0,1.3) {$v$} ;
\draw[thick] (0,1) -- (0,0.1); 
\draw[thick] (-.1,-0.1) rectangle (0.1,0.1); 
\node at (0,-.4) {$w$};
\end{tikzpicture}
\]
If we replace such a picture with the picture on the right, we arrive at a combinatorial code called framed quiver of type $\hat{A}$. 
The geometric significance is that bow varieties associated with co-balanced brane diagrams are Nakajima quiver varieties associated with the  quiver obtained this way \cite{nakajima2017cherkis}. Hence we will study which brane diagrams are Hanany-Witten equivalent to a co-balanced one.

\begin{theorem}
\label{wittentheorem}
\begin{enumerate}
\item\label{it:wittentheorem1} A standard brane diagram of type $(d,e,f)$ is Hanany-Witten equivalent to a co-balanced diagram if and only if $e$ is an {\em $m$-bounded non-decreasing integer sequence}, that is, if it satisfies $e_i\in \Z$,
    \[
e_1 \leq e_2 \leq e_3 \leq \ldots \leq e_{n-1} \leq e_n \ \leq e_1+m.
 \]
\item\label{it:wittentheorem2} In this case, the standard $(d,e,f)$ diagram is HW equivalent to a standard $(d',e',f')$ diagram with $e'$ satisfying
\[
-m < e'_1\leq e'_2 \leq \ldots \leq e'_n \leq 0.\]

\item\label{it:wittentheorem3} Let $w_l$ be the number of times $-l$ appears as a component of $e'$. Then the  Hanany-Witten equivalent quiver is
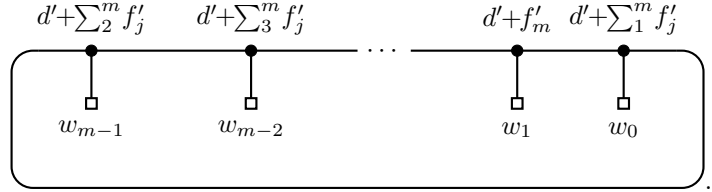
\[
\begin{tikzpicture}[scale=.7]
\draw[thick] (-2,1) -- (4,1);
\node at (4.5,1) {$\cdots$};
\draw[thick] (5,1) -- (10,1);
\draw[rounded corners=8pt,thick] (10,1) -- (10.5,1) -- (10.5,-1.6) -- (-2.5,-1.6) -- (-2.5,1) -- (-2,1);

\draw[fill] (-1,1) circle (3pt);     \node at (-1,1.6) {$d'\!\!+\!\!\sum_{2}^{m}\!f'_j$} ;
\draw[thick] (-1,1) -- (-1,0.1); 
\draw[thick] (-1.1,-0.1) rectangle (-0.9,0.1); 
\node at (-1,-.5) {$w_{m-1}$};

\draw[fill] (2,1) circle (3pt);     \node at (2,1.6) {$d'\!\!+\!\!\sum_{3}^{m}\!f'_j$} ;
\draw[thick] (2,1) -- (2,0.1); 
\draw[thick] (1.9,-0.1) rectangle (2.1,0.1); 
\node at (2,-.5) {$w_{m-2}$};

\draw[fill] (7,1) circle (3pt);     \node at (7,1.6) {$d'\!\!+\!\!f'_m$} ;
\draw[thick] (7,1) -- (7,0.1); 
\draw[thick] (6.9,-0.1) rectangle (7.1,0.1); 
\node at (7,-.5) {$w_{1}$};

\draw[fill] (9,1) circle (3pt);     \node at (9,1.6) {$d'\!\!+\!\!\sum_{1}^m\!f'_j$} ;
\draw[thick] (9,1) -- (9,0.1); 
\draw[thick] (8.9,-0.1) rectangle (9.1,0.1); 
\node at (9,-.5) {$w_{0}$};

\end{tikzpicture}.
\]
\end{enumerate}
\end{theorem}

\begin{remark} 
\label{rem:mboundedcomp}
It is obvious from the definitions of [move-1] and [move-2] that the condition $e$ is an $m$-bounded non-decreasing integer sequence is invariant under these moves. 
\end{remark}

\begin{proof}
If $e$ is an $m$-bounded non-decreasing integer sequence, then repeated applications of [move-1] or its inverse bring it to the form in Part~\eqref{it:wittentheorem2}, proving Part~\eqref{it:wittentheorem2}. 

Consider the standard $(d',e',f')$ brane diagram. Let us move $E_1$ through $-e_1$ NS5 branes to the left, using HW transitions. Then let us move $E_2$ through $-e_2$ NS5 branes to the left, using HW transitions, etc. What we obtain is exactly the brane diagram version of the quiver in Part~\eqref{it:wittentheorem3}. This proves that if $e$ is an $m$-bounded non-decreasing integer sequence, then the $(d,e,f)$ standard diagram is HW equivalent to the quiver in the figure.  

To prove the converse statement, let us consider the brane diagram version of the quiver above. The inverse of the HW transitions described above brings it to a standard form $(d',e',f')$ satisfying the condition in Part~\eqref{it:wittentheorem2}. Remark~\ref{rem:mboundedcomp} completes the proof.
\end{proof}

\begin{example} \label{ex:bowquiver}
Consider the brane diagram in standard form with $d=8$, $e=(-1,-1,0,$ $0,$ $1,1,2)$, $f=(-4,4,2)$. It is Hanany-Witten equivalent (using three inverse [move-1]'s) to $d=13$, $e'=(-2,-2,-1,-1,-1,0,0)$, $f'=(-7,1,-1)$. In turn, the associated bow variety is a quiver variety of type $\hat{A}$ with dimension vector $(13,12,6)$ and framing vector $(2,3,2)$.  
\end{example}

\subsection{D5 swaps}

\begin{definition}
Assume $d_2+d'_2 = d_1+d_3$. Carrying out the local change
\begin{equation}\label{eq:D5swap}
\begin{tikzpicture}[baseline=2]
\draw[thick] (0.1,0) -- (2.9,0);
\draw[blue,thick] (0.8,0.3) -- (1.2,-0.3);
\draw[blue,thick] (2.2,-0.3) -- (1.8,0.3);
\draw (0.5,0.3) node {$d_1$};
\draw (1.5,0.3) node {$d_2$};
\draw (2.5,0.3) node {$d_3$};
\draw[<->, thick] (4,0) -- (5,0);
\draw (4.5,0.3) node {D5 swap};
\draw[thick] (6.1,0)--(8.9,0);
\draw[blue,thick] (7.2,-0.3) -- (6.8,0.3);
\draw[blue,thick] (7.8,0.3) -- (8.2,-0.3);
\draw (6.5,0.3) node {$d_1$};
\draw (7.5,0.3) node {$d'_2$};
\draw (8.5,0.3) node {$d_3$};
\end{tikzpicture}
\end{equation}
in a brane diagram is called a D5 swap.
\end{definition}

\begin{lemma}For a brane diagram in standard $(d,e,f)$ form swapping the D5 branes $E_i$ and $E_{i+1}$ is the same as exchanging $e_i$ with $e_{i+1}$ in $e$ while keeping $d$ and $f$ fixed. 
\end{lemma}
\begin{proof} Follows from the construction.
\end{proof}

When we carry out a D5 swap on a brane diagram, the associated bow variety changes in general, even its dimension may change. In some examples the change is a homotopy equivalence (eg. one is the total space of a vector bundle with base space the other), but in other examples the cohomology ring changes considerably \cite{ji2023bow}. 


Let $X$ and $X'$ be the bow varieties associated to the diagrams on the left and right of~\eqref{eq:D5swap}. As already observed in \cite[Section~8]{rimanyi2020bow} there is a natural bijection between the torus fixed points of $X$ and $X'$. The bijection on tie diagrams is depicted in  Figure~\ref{fig:D5swap}.

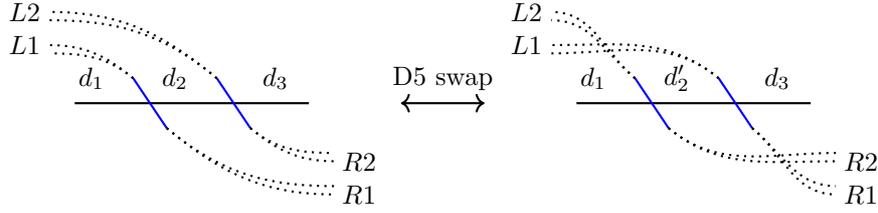
\begin{figure}
\begin{equation*}
\begin{tikzpicture}[baseline=-3,scale=1.1]
\draw[thick] (0.1,0) -- (2.9,0);
\draw[blue,thick] (0.8,0.3) -- (1.2,-0.3);
\draw[blue,thick] (2.2,-0.3) -- (1.8,0.3);
\draw (0.3,0.3) node {$d_1$};
\draw (1.3,0.3) node {$d_2$};
\draw (2.5,0.3) node {$d_3$};
\draw[dotted,thick] (0.8,0.3) to[out=140,in=0] (-.2,.6);
\draw[dotted,thick] (0.8,0.3) to[out=140,in=0] (-.2,.7);
\draw[dotted,thick] (1.8,0.3) to[out=140,in=0] (-.2,1);
\draw[dotted,thick] (1.8,0.3) to[out=140,in=0] (-.2,1.1);
\draw[dotted,thick] (1.2,-0.3) to[out=-40,in=180] (3.2,-1);
\draw[dotted,thick] (1.2,-0.3) to[out=-40,in=180] (3.2,-1.1);
\draw[dotted,thick] (2.2,-0.3) to[out=-40,in=180] (3.2,-.6);
\draw[dotted,thick] (2.2,-0.3) to[out=-40,in=180] (3.2,-.7);
\draw (-.5,1.1) node {$L2$};
\draw (-.5,.7) node {$L1$};
\draw (3.5,-1.1) node {$R1$};
\draw (3.5,-.7) node {$R2$};
\draw[<->, thick] (4,0) -- (5,0);
\draw (4.5,0.3) node {D5 swap};
\draw[thick] (6.1,0)--(8.9,0);
\draw[blue,thick] (7.2,-0.3) -- (6.8,0.3);
\draw[blue,thick] (7.8,0.3) -- (8.2,-0.3);
\draw (6.3,0.3) node {$d_1$};
\draw (7.3,0.3) node {$d'_2$};
\draw (8.5,0.3) node {$d_3$};
\draw[dotted,thick] (6.8,0.3) to[out=140,in=0] (5.8,1);
\draw[dotted,thick] (6.8,0.3) to[out=140,in=0] (5.8,1.1);
\draw[dotted,thick] (7.8,0.3) to[out=140,in=0] (5.8,.6);
\draw[dotted,thick] (7.8,0.3) to[out=140,in=0] (5.8,.7);
\draw[dotted,thick] (7.2,-0.3) to[out=-40,in=180] (9.2,-.6);
\draw[dotted,thick] (7.2,-0.3) to[out=-40,in=180] (9.2,-.7);
\draw[dotted,thick] (8.2,-0.3) to[out=-40,in=180] (9.2,-1.1);
\draw[dotted,thick] (8.2,-0.3) to[out=-40,in=180] (9.2,-1);
\draw (5.5,1.1) node {$L2$};
\draw (5.5,.7) node {$L1$};
\draw (9.5,-1.1) node {$R1$};
\draw (9.5,-.7) node {$R2$};
\end{tikzpicture}
\end{equation*}
\caption{The effect of a D5 swap on tie diagrams encoding torus fixed points. It is instructive to verify the consistency of these diagrams with the condition $d_2+d'_2=d_1+d_3$. In the language of the Maya diagrams the bijection between the fixed points is swapping the two relevant rows.}
\label{fig:D5swap}
\end{figure}

Hence, when counting torus fixed points for a $(d,e,f)$ diagram, we are allowed to carry out [move-1], [move-2], as well as permuting the components of $e$. In particular, Theorem~\ref{wittentheorem} implies that for any bow variety there is a quiver variety whose torus fixed point count is the same. 

\medskip

In Section \ref{sec:D5again} we will present a more refined analysis of the effect of D5 swap on the brane diagram. 

\begin{remark}
    The torus fixed point count is also invariant under the analogous NS5 brane swaps \cite[Section~8]{rimanyi2020bow}---imagine Figure~\ref{fig:D5swap} upside down. The brane diagrams of quiver varieties with $w=(2,3,2)$ and $v$ being one of 
    $
    (10,7,3), (10,9,3), (2,7,3), (2,1,3), (4,9,3), (4,1,3)
    $
    can be achieved from each other by HW transitions and NS5 brane swaps. Hence the fixed point count of these quiver varieties are the same. 
\end{remark}

\section{Generating series and core-quotient decomposition}
\label{sec:gen_series}


Recall that $\mathcal{M}(d,e,f)$ denotes the bow variety associated with the brane diagram in standard $(d,e,f)$ form.

Denote by $K_0(Var)$ the Grothendieck ring of varieties over $\CC$, and let $q$ be a formal variable. Consider the generating series of the classes of the bow varieties for various values of $d$:
\[ \CZ(q)=\CZ_{e,f}(q)=\sum_{d \geq 0} [\mathcal{M}(d,e,f)]q^d \in K_0(Var)\llbracket q \rrbracket. \]
\begin{remark} In a similar manner, one could introduce a multivariable generating series 
\[\CZ_{e,f}(q_{-m+1},\dots,q_n) =\sum_{d \geq 0} [\mathcal{M}(d,e,f)]q_{-m+1}^{d_{-m+1}}\cdots q_{n}^{d_{n}}\]
with $m+n$ formal variables to keep track of all multiplicities. However,  in this paper, we will not explore this multivariable version.
\end{remark}

In enumerative problems one often considers a motivic measure, which is a ring homomorphism $K_0(Var) \to R$ into a commutative ring $R$. Some well-known examples of motivic measures include the Euler charactestic (with values in $\ZZ$), the virtual Poincaré polynomial (with values in the one-variable polynomial ring $\ZZ[t]$) and the E-polynomial (with values in the two-variable polynomial ring $\ZZ[u,v]$).
A motivic measure $\phi$ induces a ring homomorphism \[\phi: K_0(Var)\llbracket q \rrbracket \to R\llbracket q \rrbracket.\] 

\begin{definition} 
Let 
\[
\begin{array}{rlll}
Z(q)= & \chi(\CZ)(q)=& \sum_{d \geq 0} \chi(\mathcal{M}(d,e,f))q^d& \in \ZZ \llbracket q \rrbracket, \\
Z(q,t)=& P_t(\CZ)(q)=& \sum_{d \geq 0} P_t(\mathcal{M}(d,e,f))q^d & \in \ZZ[t] \llbracket q \rrbracket
\end{array}
\]
denote the generating series of the Euler numbers, and the generating series of the (virtual) Poincaré polynomials associated with the Borel-Moore homology.
\end{definition}

In Section~\ref{sec:refinedgenseries}, we will derive closed formulas for $\CZ(q)$ and $Z(q,t)$ in various cases. However, first, as warm-up, let us calculate $Z(q)$ by introducing a core-quotient type combinatorial decomposition of generalized Maya diagrams. 

Let $M$ be a Maya diagram, consisting of blocks $M^{b}, b \in \frac{2\ZZ+1}{2}$. Suppose that  
\[M^{\frac{k}{2}}_{ij}=1
\qquad\qquad\text{and}\qquad\qquad
M^{\frac{k+2}{2}}_{ij}=0.\]
We obtain another Maya diagram if we replace the entry 1 by 0 and the entry 0 by 1 at the $ij$-positions of $M^{\frac{k}{2}}$ and $M^{\frac{k+2}{2}}$. When drawn on the $n \times \infty$ matrix, this corresponds to shifting the 1 from $M^{\frac{k}{2}}_{ij}$ to the right while keeping its coordinate in the $n \times m$ arrangement. The \emph{core} of a Maya diagram $M$
is the Maya diagram obtained from $M$ by successively shifting all 1's to the right using such moves at all entries, until this is no longer possible at any entry. 
\begin{lemma}
Shifting a 1 right by one block has the following effects on the dimension vectors of the diagram:
\begin{enumerate}
    \item $e$ does not change
    \item $f$ does not change
    \item $d$ decreases by 1
\end{enumerate}
\end{lemma}
\begin{proof} This follows directly from properties (i)--(iv) of Maya diagrams (end of Section~\ref{sec:maya}).
\end{proof}
We denote by $C(e,f)$ the set of $(e,f)$-core diagrams, and by 
\[\mathrm{core}\colon M(e,f) \to C(e,f)\] the map which takes an $(e,f)$-diagram to its core. 

In any $(e,f)$-core there is a unique integer $c_{ij}$  for each $1 \leq i \leq n$, $1 \leq j \leq m$ such that the leftmost 1 at the $ij$ entry occurs in the block \[\frac{2c_{ij}+1}{2} \in \frac{2 \ZZ+1}{2}.\] We associate the integer matrix  
$
c=(c_{ij})_{1 \leq i \leq n,\, 1 \leq j \leq m}
$ 
with each core. 
We clearly have
\[ \sum_{j=1}^m c_{ij}=e_i, \quad 1 \leq i \leq n 
\qquad\qquad\text{and}\qquad\qquad
 \sum_{i=1}^m c_{ij}=f_j, \quad 1 \leq j \leq m. \]
\begin{example}
    The core of the first Maya diagram in Example \ref{ex:2Maya} and the associated $(c_{ij})$ matrix are: 
    \[
\begin{tikzpicture}
\draw (-2.2,-.2) -- (3.4,-0.2);\draw (-2.2,1) -- (3.4,1);
\draw (-.2,-0.5) -- (-.2,1.3);
\draw (0.2,1.4) node {$\frac{1}{2}$};
\draw (1,1.4) node {$\frac{3}{2}$};
\draw (1.8,1.4) node {$\frac{5}{2}$};
\draw (2.6,1.4) node {$\frac{7}{2}$};
\draw (-.6,1.4) node {$\frac{-1}{2}$};
\draw (-1.4,1.4) node {$\frac{-3}{2}$};
 \draw (0,0) node {$1$};
 \draw (.4,0) node {$0$};
 \draw (0,.4) node {$0$};
 \draw (.4,.4) node {$1$};
 \draw (0,.8) node {$0$};
 \draw (.4,.8) node {$0$};
 \draw (-.2,-.2)--(0.6,-0.2) -- (0.6,1) -- (-0.2,1) -- (-.2,-.2);
 \draw (0.8,0) node {$1$};
 \draw (1.2,0) node {$0$};
 \draw (0.8,.4) node {$0$};
 \draw (1.2,.4) node {$1$};
 \draw (0.8,.8) node {$1$};
 \draw (1.2,.8) node {$1$};
 \draw (.6,-0.2)--(1.4,-0.2) -- (1.4,1) -- (0.6,1) -- (0.6,-.2);
 \draw (1.6,0) node {$1$};
 \draw (2,0) node {$1$};
 \draw (1.6,.4) node {$1$};
 \draw (2,.4) node {$0$};
 \draw (1.6,.8) node {$1$};
 \draw (2,.8) node {$1$};
 \draw (1.4,-0.2)--(2.2,-0.2) -- (2.2,1) -- (1.4,1) -- (1.4,-.2);
 \draw (2.4,0) node {$1$};
 \draw (2.8,0) node {$1$};
 \draw (2.4,.4) node {$1$};
 \draw (2.8,.4) node {$1$};
 \draw (2.4,.8) node {$1$};
 \draw (2.8,.8) node {$1$};
 \draw (2.2,-0.2)--(3,-0.2) -- (3,1) -- (2.2,1) -- (2.2,-.2);
 \draw (-.8,0) node {$1$};
 \draw (-.4,0) node {$0$};
 \draw (-0.8,.4) node {$0$};
 \draw (-.4,.4) node {$0$};
 \draw (-0.8,.8) node {$0$};
 \draw (-.4,.8) node {$0$};
 \draw (-1,-.2)--(-0.2,-0.2) -- (-0.2,1) -- (-1,1) -- (-1,-.2);
 \draw (-1.6,0) node {$0$};
 \draw (-1.2,0) node {$0$};
 \draw (-1.6,.4) node {$0$};
 \draw (-1.2,.4) node {$0$};
 \draw (-1.6,.8) node {$0$};
 \draw (-1.2,.8) node {$0$};
 \draw (-1.8,-.2)--(-1,-0.2) -- (-1,1) -- (-1.8,1) -- (-1.8,-.2);
\draw (3.4,.5) node {$\forall 1$};
\draw (-2.2,.5) node {$\forall 0$};
\draw ( 4,0) node {$E_3$};
\draw ( 4,.4) node {$E_2$};
\draw ( 4,.8) node {$E_1$};
\draw (0,-0.6) node {$F_1$};
\draw (0.4,-0.6) node {$F_2$};
\end{tikzpicture}
\qquad\qquad
\begin{tikzpicture}
 \draw (0,0) node {$-1$};
 \draw (.4,0) node {$2$};
 \draw (0,.4) node {$2$};
 \draw (.4,.4) node {$1$};
 \draw (0,.8) node {$1$};
 \draw (.4,.8) node {$1$};
 \draw (-.2,-.2)--(0.6,-0.2) -- (0.6,1) -- (-0.2,1) -- (-.2,-.2);
 \draw (1.4,0) node {$e_3=1$};
 \draw (1.4,0.4) node {$e_2=3$};
 \draw (1.4,0.8) node {$e_1=2$};
 \node[rotate=90] at (0,1.6) {$f_1=2$};
 \node[rotate=90] at (0.4,1.6) {$f_2=4$};
\end{tikzpicture}.
\]
\end{example}

Conversely, any $n \times m$ integer matrix completely determines a core, so we get a bijection
\begin{equation}\label{typeA-cores} 
C(e,f) \longleftrightarrow \left\{(c_{ij}):\, \sum_j  c_{ij}=e_i, 1 \leq i \leq n,  \, \sum_{i} c_{ij}=f_j, 1 \leq j \leq m \right\}\subset \ZZ^{n \times m}.\end{equation}
\begin{lemma} 
\label{lem:dij}
For the $(e,f)$-core corresponding to the matrix $(c_{ij})$, we have 
\[ 
d= \sum_{i,j} \frac{c_{ij}(c_{ij}-1)}{2}.
\]
\end{lemma}
\begin{proof}
Consider the contribution of the $ij$-entries to the expression for $d$ appearing in property (iv) above. For $c_{ij} > 0$, the contribution is $1+2+\ldots+(c_{ij}-1)=\binom{c_{ij}}{2}$. For $c_{ij}\leq 0$, the contribution is $1+2+\ldots+(-c_{ij})=\binom{-c_{ij}+1}{2}=\binom{c_{ij}}{2}$.
\end{proof}

\begin{lemma} There is a bijection
\[ M(e,f) \longleftrightarrow C(e,f) \times P^{nm} \]
where $P$ is the set of partitions. (That is, $P^{nm}$ is the set of ordered $nm$-tuples of partitions.)
\end{lemma}

\begin{proof} For any $(e,f)$-diagram the positions of the 1's at the $ij$-entries  determine a partition: the leftmost 1 is shifted left from the leftmost 1 of its core by a nonnegative integer $\lambda_1$, the second leftmost 1 is shifted left from the second leftmost 1 of its core by a nonnegative integer $\lambda_2 \leq \lambda_1$, etc.
\end{proof}

Thus we obtain the following expression for 
$Z(q)=Z_{e,f}(q)=\sum_{d \geq 0} |M(d,e,f)|q^d$.

\begin{theorem}
\label{thm:eulerchargen}  Let $e \in \ZZ^{n}$ and $f \in \ZZ^{m}$ be fixed vectors. Then
\[
Z(q) =  
\underbrace{
\left( 
\sum_{\substack{c \in \ZZ^{nm} \\ \sum_j c_{ij}=e_i \\ \sum_{i}c_{ij}=f_j }}
q^{\sum_{i,j} c_{ij}(c_{ij}-1)/2}
\right)
}_{Z_0(q)}
\cdot 
\prod_{l=1}^{\infty} \left(\frac{1}{1-q^l}\right)^{nm}. \]
\end{theorem}

The summation runs for a translated lattice in $\R^{nm}$, with the $q$-exponent being a quadratic function of the coordinates. Thus, the series can be viewed as translated lattice theta-function, that are expected to exhibit modular properties. 

\begin{example}
    For $e=(-1,1)$, $f=(0,0)$ the series in the large parentheses is 
    \[
    Z_0(q)=2\sum_{k=0}^\infty 
    q^{2k(k+1)+1}
    \]
    and hence 
    $
    Z(q)=2q+8q^2+28q^3+80q^4+212q^5+512q^6+1176q^7+O(q^8)
    $.
\end{example}
\begin{example} \label{example:321321}
   For more interesting $e$ and $f$ vectors computer evidence suggests that we obtain series of number theoretical significance, consistent with physics expectations. Here are two examples. 
   \begin{itemize}
    \item For $e=f=(3,2,1)$ the series in the large parentheses is
    \[ Z_0(q)=1+7q+8q^2+18q^3+14q^{4}+31q^{5}+20q^{6}+36q^{7}+O(q^8)
    \]
    By Theorem~\ref{thm:mod}, the coefficients are exactly $\sigma(3n+1)$ ($\sigma$ is the sum-of-divisors function). This series---up to a reparametrization---is a modular form for the congruence subgroup $\Gamma_0(9)$ of weight 2 (cf. A144614 of \cite{OEIS}). The infinite product part is equal to $q^{\frac{9}{24}}\eta(q)^{-9}$ where $\eta(q)=\eta(e^{2 \pi i \tau})$ is the Dedekind $\eta$-function. Therefore, the series
    \[
    Z(q)=1+16q+125 q^2+723 q^3+ 3428 q^{4}+ 14167 q^{5} + 52679 q^{6}+O(q^7)
    \]
    is a meromorphic modular form.
    \item The $e=(-1,-1,-1)$, $f=(-2,-1,0)$ example is similar: we have
    \[ 
    Z_0(q)=3q^3+6q^4+15q^5+12q^6+24q^{7}+18q^{8}+42q^{9}+24q^{10}+O(q^{11})
    \]
    whose coefficients are $\sigma(3n+2)$.
\end{itemize}
\end{example}

\begin{remark} \rm
  It is easy to verify that the effects of [D5-swap], [move-1], or [move-2] on the series $Z(q)$ are inessential: $Z(q)$ gets multiplied by a $q$-power ($q^0$, $q^{e_1}$, $q^{f_m}$, respectively). 
\end{remark}


\begin{remark} 
If $e=(0,\ldots,0)$ then $\mathcal{M}(d,e,f)$ is a quiver variety of affine type A such that the framing vector has one non-0 component. Various formulas are know for $Z(q)$ in this case. For example, consider one more summation over the possible dimension vectors $f$ 
\[  \sum_{\substack{f \in \ZZ^{m} \\ \sum_jf_j =0}}  \sum_{\substack{c \in \ZZ^{nm} \\ \sum_j c_{ij}=0 \\ \sum_{i}c_{ij}= f_j }}q^{\sum_{i,j} c_{ij}(c_{ij}-1)/2}.
\]
This is equal to
\[
\sum_{\substack{c \in \ZZ^{nm} \\ \sum_j c_{ij}=0}}
q^{\sum_{i,j} c_{ij}(c_{ij}-1)/2}
 = \sum_{\substack{l \in \ZZ^{n(m-1)}}} q^{
\sum_{i=1}^{n} \frac{1}{2}\langle l_i,  l_i\rangle_{\Delta}}
\]
where the inner product is taken with respect to the Cartan matrix of finite type $A_{m-1}$; see eg. \cite[Section~3]{gyenge2017enumeration}, cf. \cite[Corollary~4.12]{fujii2017combinatorial}.
\end{remark}

\section{Equivariant K-theory of the tangent space}
\label{sec:equivK}

Recall that $T=\C^*_{t_1} \times \C^*_{t_2} \times \left(\C^*\right)^{\text{D5 branes}}$, and correspondingly, we have 
\[
K_T(\pt)=\Z[t_1^{\pm1},t_2^{\pm 1},u_1^{\pm 1},u_2^{\pm 1},\ldots, u_n^{\pm 1}].\]
The tangent space of a $T$-variety at a fixed point represents an element in this ring. For example, the bow variety $\mathcal{M}(5,(0,3),(-1,4))$ (of dimension~6) has five torus fixed points. The tangent spaces at these five fixed points represent 
\begin{equation}\label{03-14-example}
\begin{array}{lll}
   (u_1u_2^{-1})\left( t_1^{-4}+t_1^{-3}t_2^{-1}+t_1^{-2}\right)& +
    (u_2u_1^{-1})\left(t_1^5t_2+t_1^4t_2^2+t_1^3t_2\right),&  \\
    (u_1u_2^{-1})\left( t_1^{-2}t_2^2+t_1^{-2}\right)& +
    (u_2u_1^{-1})\left(t_1^3t_2^{-1}+t_1^3t_2\right) & + (t_1t_2^{-1}+t_2^2),\\
    (u_1u_2^{-1})\left( t_1^{-1}t_2+t_1^{-2}\right)& +
    (u_2u_1^{-1})\left(t_1^2+t_1^3t_2 \right) & + (t_1^{-1}t_2+t_1^2),\\
    (u_1u_2^{-1})\left( t_1^{-2}+t_1^{-3}t_2\right)& +
    (u_2u_1^{-1})\left(t_1^3t_2+t_1^4\right) & + (t_1t_2^{-1}+t_2^2),\\
    (u_1u_2^{-1})\left( 2t_1^{-2}\right)& +
    (u_2u_1^{-1})\left(2 t_1^3t_2\right) & + (t_1^{-1}t_2+t_1^2)
\end{array}
\end{equation}
in $K_T(\pt)$. Our goal in this section is to develop combinatorial recipes to find such K-theory classes of tangent spaces at fixed points of $\mathcal{M}(d,e,f)$. We will obtain the result in two forms: Theorems~\ref{thm:sum01pairs} and~\ref{thm:tangentcharm}.
    
\subsection{The tangent space at a fixed point}

For each D3 brane $X$ the decomposition from Section~\ref{subsec:torus_fixed_points} restricted to $V_X$  gives rise to a decomposition
\[ V_X= \bigoplus_{i=1}^n V_{X,i}.\]
For $1 \leq \alpha, \beta \leq n$ we define
\[ \mathbb{M}_{\alpha,\beta}=\bigoplus_{E\; \mathrm{D5}} \mathbb{M}_{E,\alpha,\beta} \oplus \bigoplus_{F\; \mathrm{NS5}} \mathbb{M}_{F,\alpha,\beta}, \]
\[ \mathbb{N}_{\alpha,\beta}=\bigoplus_{E\; \mathrm{D5}} \mathbb{N}_{E,\alpha,\beta} \oplus \bigoplus_{X\; \mathrm{D3}} \mathbb{N}_{X,\alpha,\beta} \]
using
\[ 
\begin{aligned}
\mathbb{M}_{E,\alpha,\beta} =\; & \mathrm{Hom}(V_{E^+,\alpha},V_{E^-,\beta}) \oplus t_1t_2 \mathrm{Hom}(V_{E^+,\alpha},\CC_{E,\beta}) \\& \oplus \mathrm{Hom}(\CC_{E,\alpha},V_{E^-,\beta}) \oplus t_1t_2 \mathrm{Hom}(V_{E^-,\alpha},V_{E^-,\beta}) \oplus t_1t_2 \mathrm{Hom}(V_{E^+,\alpha},V_{E^+,\beta}), \\
\mathbb{M}_{F,\alpha,\beta}=\; &t_1\mathrm{Hom}(V_{F^+,\alpha},V_{F^-,\beta}) \oplus t_2\mathrm{Hom}(V_{F^-,\alpha},V_{F^+,\beta}) ,\\ 
 \mathbb{N}_{E,\alpha,\beta} =\; & t_1t_2\mathrm{Hom}(V_{E^+,\alpha},V_{E^-,\beta}),\\ 
 \mathbb{N}_{X,\alpha,\beta} =\; & t_1t_2\,\mathrm{Hom}(V_{X,\alpha},V_{X,\beta}) 
\end{aligned}
\]
where $E,F$ and $X$ are D5, NS5 and D3 branes respectively.

\begin{corollary} 
\label{cor:kertaumodimsigma}
The quotient $\mathrm{Ker}\, \tau / \mathrm{Im}\, \sigma$ decomposes as 
\[ \mathrm{Ker}\, \tau / \mathrm{Im}\, \sigma = \oplus_{\beta,\alpha} (\mathrm{Ker}\, \sigma_{\alpha,\beta}/ \mathrm{Im} \, \tau_{\beta,\alpha}) \u_{\beta}\u_{\alpha}^{-1} \]
where
\[ \bigoplus_{X\; \mathrm{D3}} \mathrm{Hom}(V_{X,\alpha},V_{X,\beta}) \xrightarrow{\sigma_{\beta,\alpha}}\mathbb{M}_{\alpha,\beta}\xrightarrow{\tau_{\beta,\alpha}=\begin{pmatrix} \tau_{\beta,\alpha,1} \\ \tau_{\beta,\alpha,2} \end{pmatrix}} \mathbb{N}_{\alpha,\beta}. \]
\end{corollary}
\begin{proof} Follows from 
Lemma~\ref{lem:complex_is_equivariant}, the fact that the moment map is $T$-equivariant \cite[Lemma~3.2]{rimanyi2020bow},
and the construction of the bow variety.
\end{proof}

\subsection{Combinatorial codes of torus fixed points---extended Young diagrams}

We continue to assume that there are $n$ D5 branes and $m$ NS5 branes. In this section we present yet another description of torus fixed points of bow varieties, which will be useful in some of our calculations.

Let $M=(M_1,\dots,M_n)$ be a generalized Maya diagram with component $M_i$ corresponding to the D5 brane $E_i$. We can represent each row $M_i$ with a Young diagram type picture. 
Consider the set ${\mathbb Z}\times {\mathbb Z}$ of integer pairs, which we will visualize as a set of blocks on the plane rotated by $45$ degrees. Specifically, axis for the $t_1$-weight runs from top left to bottom right, while the axis for the $t_2$-weight runs from top right to bottom left.

Since we are only considering integer coordinates, vertical lines are of the form $(a-b) \mod m \equiv j$.
Blocks along the line $(a-b) \mod m \equiv 0$ will be further divided into two triangles vertically. In each such block the left half is labelled 0 while the right half is labelled 1.
The remaining blocks are labelled vertically with $m-1$ labels $2, \ldots, m$ as shown in the picture: the block at position $(i,j)$ is labelled with $(a-b)+1 \mod m$.

For example, the labeling in the top quadrant looks as follows:
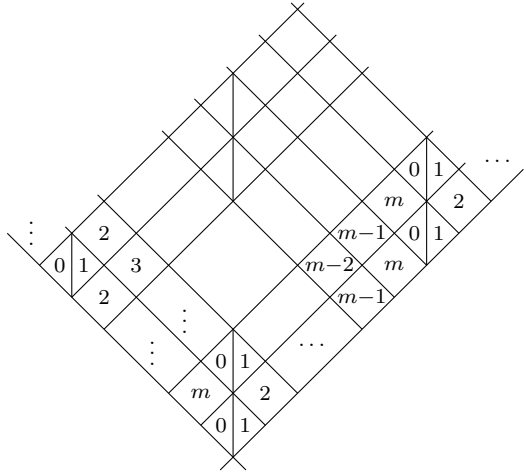
\begin{center}
\begin{tikzpicture}[scale=0.6, font=\footnotesize, fill=black!20,rotate=45]
  \draw (0, -0.4) -- (0,7);
  \foreach \x in {1,2,4,5,6,7,8}
    {
      \draw (\x, 0) -- (\x,6.2);
    }
    \draw (-0.4,0) -- (9,0);
   \foreach \y in {1,2,4,5,6}
    {
         \draw (0,\y) -- (8.2,\y);
    }
    \draw (0,0) -- (2,2);
    \draw (6,0) -- (8,2);
    \draw (4,4) -- (6,6);
    \draw (0,5) -- (1,6);
    \draw (0.5,0.5) node {$0\phantom{..} 1$};
    \draw (1.5,0.5) node {2};
    \draw (4.5,0.5) node {$m$$-$$1$};
    \draw (5.5,0.5) node {$m$};
    \draw (0.5,1.5) node {$m$};
    \draw (1.5,1.5) node {$0 \phantom{..} 1$};
    \draw (4.5,1.5) node {$m$$-$$2$};
    \draw (5.5,1.5) node {$m$$-$$1$};
    \draw (6.5,0.5) node {$0 \phantom{..} 1$};
    \draw (7.5,0.5) node {2};
    \draw (6.5,1.5) node {$m$};
    \draw (7.5,1.5) node {$0\phantom{..} 1$};
    
    \draw (0.5,4.5) node {2};
    \draw (1.5,4.5) node {3};
    \draw (0.5,5.5) node {$0\phantom{..} 1$};
    \draw (1.5,5.5) node {2};
    \draw(0.5,3) node {\vdots};
    \draw(1.5,3) node {\vdots};
    \draw(3,0.5) node {\dots};
    \draw(8.75,0.5) node {\dots};
    \draw(0.5,6.75) node {\vdots};
\end{tikzpicture}
\end{center}
This diagonal labeling (or coloring) is a modification of the pattern of affine type~$\hat{A}$ \cite{gyenge2017enumeration}.
When drawing a diagram, we understand that its boxes are colored (labelled) according to the above pattern.



In the classical Young diagram literature the number $e_i$
is called the charge of $M_i$. When $e_i=0$ there is a well defined $A_{m-1}$-colored Young diagram corresponding to $M_i$; see e.g. \cite[Section~2.5]{nagao2009quiver} or \cite[Section~2.3]{fujii2017combinatorial}. We generalise this to to cases where $e_i$ is arbitrary. Set the base level of our diagram to the line $y=e_i$ (recall that the picture is rotated by 45 degrees):

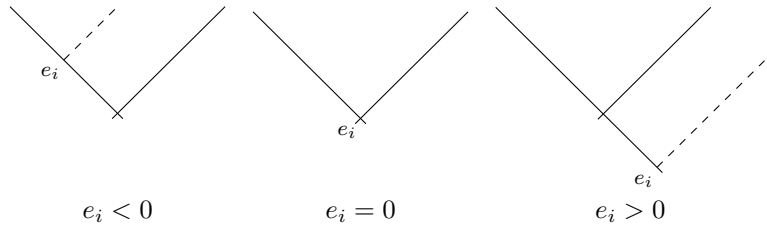
\begin{center}
\begin{tabular}{c c c}
\begin{tikzpicture}[scale=0.5, font=\footnotesize, fill=black!20,rotate=45]
    \draw (0, -0.2) -- (0,4);
    \draw (-0.2, 0) -- (4,0);
    \draw[dashed] (0, 2) -- (2,2);
    \draw(-0.5,2) node {$e_i$};
\end{tikzpicture}&
\begin{tikzpicture}[scale=0.5, font=\footnotesize, fill=black!20,rotate=45,baseline=-0.3]
    \draw (0, -0.2) -- (0,4);
    \draw (-0.2, 0) -- (4,0);
    \draw(-0.5,0) node {$e_i$};
\end{tikzpicture}&
\begin{tikzpicture}[scale=0.5, font=\footnotesize, fill=black!20,rotate=45,baseline=-2]
    \draw (0, -2.2) -- (0,4);
    \draw (-0.2, 0) -- (4,0);
    \draw[dashed] (0, -2) -- (4,-2);
    \draw(-0.5,-2) node {$e_i$};
\end{tikzpicture} \\
$e_i < 0$ & $e_i = 0$ & $e_i > 0$
\end{tabular}
\end{center}

The diagram representing $M_i$ is drawn as follows:
\begin{itemize}
    \item Draw the Young diagram $Y_i$ associated with $M_i$ as in \cite[Section~2.5]{nagao2009quiver} but in the quadrant where $x \geq 0$ and $y \geq e_i$. 
    \item Extend this with the triangle 
    \begin{equation} 
    \label{eq:Tdef}
    T_i = T_{e_i} \coloneqq 
    \begin{cases}
    (0,0) \text{--} (e_i+1,0) \text{--} (e_i+1,e_i+1) & \quad  \text{if}\quad e_i < 0,\\
    (1,1) \text{--} (e_i,1) \text{--} (e_i,e_i) & \quad  \text{if}\quad e_i > 0.
    \end{cases}\end{equation}
\end{itemize}
In this way, we obtain a diagram \[B_i = Y_i+T_i\] where $Y_i$ and $T_i$ are the Young diagram and triangle parts drawn in steps (1) and (2), respectively. These pictures are \emph{extended Young diagrams}. 
Here are three typical examples:
\vspace{0.5cm}
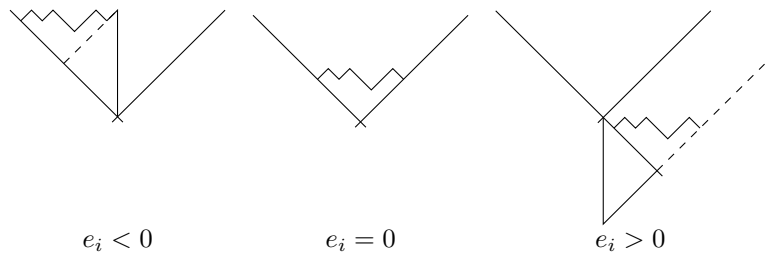
\begin{center}
\begin{tabular}{c c c}
\begin{tikzpicture}[scale=0.5, font=\footnotesize, fill=black!20,rotate=45]
    \draw (0, -0.2) -- (0,4);
    \draw (-0.2, 0) -- (4,0);
    \draw[dashed] (0, 2) -- (2,2);
    \draw (0,3.6) -- (0.4,3.6) -- (0.4,3.2) -- (0.8,3.2) -- (0.8,2.4) -- (1.6,2.4) -- (1.6,2);
    \draw (1.6,2) -- (2,2) -- (0,0);
\end{tikzpicture}&
\begin{tikzpicture}[scale=0.5, font=\footnotesize, fill=black!20,rotate=45,baseline=-0.3]
    \draw (0, -0.2) -- (0,4);
    \draw (-0.2, 0) -- (4,0);
    \draw (0,1.6) -- (0.4,1.6) -- (0.4,1.2) -- (0.8,1.2) -- (0.8,0.4) -- (1.6,0.4) -- (1.6,0);
\end{tikzpicture}&
\begin{tikzpicture}[scale=0.5, font=\footnotesize, fill=black!20,rotate=45,baseline=-2]
    \draw (0, -2.2) -- (0,4);
    \draw (-0.2, 0) -- (4,0);
    \draw[dashed] (0, -2) -- (4,-2);
    \draw (0,-0.4) -- (0.4,-0.4) -- (0.4,-0.8) -- (0.8,-0.8) -- (0.8,-1.6) -- (1.6,-1.6) -- (1.6,-2);
    \draw (0, 0) -- (-2,-2) -- (0,-2); 
\end{tikzpicture} \\
$e_i < 0$ & $e_i = 0$ & $e_i > 0$
\end{tabular}
\end{center}
\vspace{0.3cm}
Note that the extended Young diagram $Y_i +T_i$ is uniquely determined by the pair $(Y_i,e_i)$.
\begin{corollary}
Maya diagrams $(M_1,\dots,M_n)$ correspond bijectively to $n$-tuples $(B_1,\dots,B_n)$ of extended Young diagrams.
\end{corollary}

Fix an extended Young diagram $B=Y+T$. This diagram corresponds to a brane configuration with a single D5 brane and $m$ NS5 branes. The associated bow variety has exactly $m+1$ D3 branes, labelled by $0, \dots,m$. Denote by $B^l$ the blocks in $B$ on the diagonal $a- b =l$. The vector space associated with the $j$-th D3 brane is then
\[ V_j= \bigoplus_{k \in \Z}V_{j,\frac{2k-1}{2}}\]
where 
\[V_{j,\frac{2k-1}{2}} =\sum_{s \in B^{m+j-1}} t_1^{-s_1}t_2^{-s_2} \]
if $(j,k) \not\in \{(0,0),(1,1)\}$. For $(j,k) \in \{(0,0),(1,1)\}$ the corresponding spaces are
\[ V_{0,\frac{-1}{2}}=\sum_{s \in B^{0}} t_1^{-s_1}t_2^{-s_2} \quad \textrm{and} \quad V_{1,\frac{1}{2}}=\sum_{s \in Y^{0}} t_1^{-s_1}t_2^{-s_2} \]
if $e \leq 0$
or
\[ V_{0,\frac{-1}{2}}=\sum_{s \in Y^{0}} t_1^{-s_1}t_2^{-s_2} \quad \textrm{and} \quad V_{1,\frac{1}{2}}=\sum_{s \in B^{0}} t_1^{-s_1}t_2^{-s_2} \]
if $e \geq 0$.
The dimension of $V_j$, $0 \leq j \leq m$ is then
\[  \mathrm{wt}_j(B) \coloneqq \dim V_j\]
which can also be read off from $B$ as the total number of half or full blocks of label $j$ in $B$. 

More generally, for a Maya diagram $M=(M_1,\dots,M_n)$ with $n$ D5 branes the dimension of $V_j$ is
\[ d_j=\dim V_j = \sum_{i=1}^n\mathrm{wt}_j(B_i).\]


\subsection{Torus characters of the tangent space}

\subsubsection{Case of one NS5 brane}
In this section, we assume $m=1$. 
We have observed that if the charge $e_i$ of the row $M_i$ of a Maya diagram is zero, the corresponding extended Young diagram is just a standard Young diagram. When all $e_i=0$, then the bow variety becomes a quiver variety. We will reduce the computation of the characters for an arbitrary extended Young diagram inductively to this quivers case by gradually decreasing $\sum_i|e_i|$ until it reaches zero.

If the charge $e_i < 0$,
we can write an extended Young diagram $B_i$ as
\[ B_i = (0,0)+(-1,-1)+\dots+(e_i+1,e_i+1) + (-1,0)\cdot\overline{B}_i\]
where $\overline{B}_i$ is an extended Young diagram with charge $\overline{e}_i=e_i+1$. Then $Y_i=(-1,0) \cdot \overline{Y}_i$ for the Young diagram parts, and the monomials
\[ L_0=1+\dots+t_1^{e_i+1}t_2^{e_i+1} \]
resp.
\[ L_1=t_1^{-1}(1+\dots+t_1^{e_i+2}t_2^{e_i+2})\]
span the half-boxes which were deleted from $B_i$.
Similarly, if $e_i > 0$, write $B_i$ as
\[ B_i= (1,1)+(2,2)+\dots+(e_i,e_i) + (1,0)\cdot\overline{B}_i\]
where $\overline{B}$ is an extended Young diagram with charge $\overline{e}_i=e_i-1$. In this case $Y_i=(1,0) \cdot \overline{Y}_i$, and the monomials
\[ L_1=t_1t_2+\dots+t_1^{e_i}t_2^{e_i} \]
resp.
\[ L_0=t_1(t_1t_2+\dots+t_1^{e_i-1}t_2^{e_i-1})\]
span the half-boxes removed from $B_i$.


Fix $1 \leq \alpha, \beta \leq n$ with diagrams $B_{\alpha} = Y_{\alpha}+T_{\alpha}$ and $B_{\beta} = Y_{\beta}+T_{\beta}$.
Denote by
\[ N_{\alpha,\beta}(t_1,t_2) \coloneqq \mathrm{Ker}\, \sigma_{\beta, \alpha}/ \mathrm{Im} \, \tau_{\beta,\alpha},\]
the contribution of the pair $(B_{\alpha},B_{\beta})$ to the character in Corollary~\ref{cor:kertaumodimsigma}.
Denote by $N_{\overline{\alpha},\beta}(t_1,t_2)$, resp. $N_{\alpha,\overline{\beta}}(t_1,t_2)$ the same quotient for the pair $(\overline{B}_{\alpha},B_{\beta})$, resp. $(B_\alpha,\overline{B}_\beta)$.

\begin{lemma}[Reduction lemma] Let $m=1$. For $\alpha \leq \beta$,
\[ N_{\alpha,\beta}=
\begin{cases}
   t_1 N_{\overline{\alpha},\beta} + t_1t_2L_{1,\alpha}^{\ast} + t_1^{-e_{\alpha}}t_2^{-e_{\alpha}-1}(V_{1,\beta}-V_{0,\beta}) & \textrm{if} \quad e_{\alpha} < 0, \\
   t_1^{-1} N_{\overline{\alpha},\beta} + t_1t_2L_{1,\alpha}^{\ast} + t_1^{-e_{\alpha}}t_2^{-e_{\alpha}}(V_{0,\beta}-V_{1,\beta}) & \textrm{if} \quad e_{\alpha} > 0, \\
   t_1^{-1} N_{\alpha,\overline{\beta}} + L_{0,\beta} + t_1^{e_{\beta}+1}t_2^{e_{\beta}+1}(V_{1,\alpha}^{\ast}-V_{0,\alpha}^{\ast}) & \textrm{if} \quad e_{\beta} < 0, \\
   t_1 N_{\alpha,\overline{\beta}} + L_{0,\beta} + t_1^{e_{\beta}+1}t_2^{e_{\beta}}(V_{0,\alpha}^{\ast}-V_{1,\alpha}^{\ast}) & \textrm{if} \quad e_{\beta} > 0.
\end{cases}\]
For $\alpha \geq \beta$,
\[ N_{\alpha,\beta}=
\begin{cases}
   t_1 N_{\overline{\alpha},\beta} + t_1t_2L_{0,\alpha}^{\ast} + t_1^{-e_{\alpha}}t_2^{-e_{\alpha}}(V_{1,\beta}-V_{0,\beta}) & \textrm{if} \quad e_{\alpha} < 0, \\
   t_1^{-1} N_{\overline{\alpha},\beta} + t_1t_2L_{0,\alpha}^{\ast} + t_1^{-e_{\alpha}}t_2^{-e_{\alpha}+1}(V_{0,\beta}-V_{1,\beta}) & \textrm{if} \quad e_{\alpha} > 0, \\
   t_1^{-1} N_{\alpha,\overline{\beta}} + L_{1,\beta} + t_1^{e_{\beta}+1}t_2^{e_{\beta}+2}(V_{1,\alpha}^{\ast}-V_{0,\alpha}^{\ast}) & \textrm{if} \quad e_{\beta} < 0, \\
   t_1 N_{\alpha,\overline{\beta}} + L_{1,\beta} + t_1^{e_{\beta}+1}t_2^{e_{\beta}+1}(V_{0,\alpha}^{\ast}-V_{1,\alpha}^{\ast}) & \textrm{if} \quad e_{\beta} > 0.
\end{cases}\]
\end{lemma}
\begin{proof}
We just prove the first case out of the eight; the remaining cases are similar. That is, we assume that $\alpha \leq \beta$, $e_{\alpha} <0$  and we perform the reduction
\[ (B_{\alpha},B_{\beta}) \to (\overline{B_{\alpha}},B_{\beta}).\]
The total tangent space $\mathrm{Ker}\, \sigma/ \mathrm{Im} \, \tau$ can be expressed as
\[
\begin{aligned}
N(t_1,t_2) =\ & \sum_{i=1}^n\left((1-t_1t_2)V_{i}^{\ast}\otimes V_{i-1}+ t_1t_2V_{i}^{\ast}+V_{i-1} +t_1t_2 V_{i-1}^{\ast}\otimes V_{i-1} + t_1t_2 V_{i}^{\ast}\otimes V_{i}\right) \\
& + \left(t_1 V_{1}^{\ast}\otimes V_{0}+t_2 V_{0}^{\ast}\otimes V_{1}\right) 
 - \sum_{l=0}^{n} (1+t_1t_2)V_{l}^{\ast} \otimes V_{l}.
\end{aligned}\]
By the considerations of Section~\ref{subsec:torus_fixed_points} we have for $0 \leq i \leq n$ that
\[ V_i=V_{1,1}\u_1+\dots+ V_{1,i}\u_i+V_{0,i+1}\u_{i+1}+\dots+V_{0,n}\u_n.\]
Therefore,
\begin{equation}
\label{eq:Nalphabetam1}
\begin{aligned}
N_{\alpha,\beta} =\ & (1-t_1t_2)V_{1,\alpha}^{\ast}\otimes V_{0,\beta}+ t_1t_2V_{1,\alpha}^{\ast}+V_{0,\beta} + t_1 V_{0,\alpha}^{\ast}\otimes V_{1,\beta}+t_2 V_{1,\alpha}^{\ast}\otimes V_{0,\beta}  \\ 
& - V_{0,\alpha}^{\ast}\otimes V_{0,\beta} - V_{1,\alpha}^{\ast}\otimes V_{1,\beta}.
\end{aligned}\end{equation}
By our construction, we have that
\[ V_{0,\alpha}=L_{0,\alpha}+t_1^{-1}\overline{V}_{0,\alpha} \quad \textrm{and}\quad V_{1,\alpha}=L_{1,\alpha}+t_1^{-1}\overline{V}_{1,\alpha}.\]
Substituting these into \eqref{eq:Nalphabetam1}, we recover $t_1 N_{\overline{\alpha},\beta}$ plus the additional terms
\[
 (1- t_1)V_{0,\beta}+(1+t_2-t_1t_2)L_{1,\alpha}^{\ast} \otimes V_{0,\beta}+t_1t_2L_{1,\alpha}^{\ast} + t_1 L_{0,\alpha}^{\ast} \otimes V_{1,\beta}-L_{0,\alpha}^{\ast} \otimes V_{0,\beta}-L_{1,\alpha}^{\ast} \otimes V_{1,\beta}.
\]
The coefficient of $V_{1,\beta}$ in this expression is
\[ t_1L_{0,\alpha}^{\ast}-L_{1,\alpha}^{\ast}=t_1^{e_{\alpha}}t_2^{e_{\alpha}-1}.\]
Similarly, the coefficient of $V_{0,\beta}$ is
\[ 1-t_1+L_{1,\alpha}^{\ast}+t_2L_{1,\alpha}^{\ast}-t_1t_2L_{1,\alpha}^{\ast}-L_{0,\alpha}^{\ast}=-t_1^{e_{\alpha}}t_2^{e_{\alpha}-1}.\]
Collecting these terms together and noting that there remains a term of $t_1t_2L_{1,\alpha}^{\ast}$ we obtain the claim.
\end{proof}

We continue to assume that $m=1$. Denote by $N_{Y_{\alpha},Y_{\beta}}$ the contribution to the character of the Young diagram parts of $B_{\alpha}$ and $B_{\beta}$.
Define 
\begin{equation} 
\label{eq:Rdef}
R_{\alpha,\beta}^{e_{\alpha},e_{\beta}} \coloneqq 
\begin{cases}
T_{e_{\beta}-e_{\alpha}} & \textrm{if} \quad \alpha \leq \beta \textrm{ and } e_{\alpha} \geq e_{\beta}, \\
t_1T_{e_{\beta}-e_{\alpha}-1} & \textrm{if} \quad \alpha \leq \beta \textrm{ and } e_{\alpha} < e_{\beta}, \\
T_{e_{\beta}-e_{\alpha}} & \textrm{if} \quad \alpha \geq \beta \textrm{ and } e_{\alpha} \leq e_{\beta}, \\
t_1^{-1}T_{e_{\beta}-e_{\alpha}+1} & \textrm{if} \quad \alpha \geq \beta \textrm{ and } e_{\alpha} > e_{\beta}
\end{cases}
\end{equation}
where $T_i$ is defined in \eqref{eq:Tdef}. 

\begin{proposition}
\label{prop:Nalphabetared}
\[N_{\alpha,\beta}=t_1^{e_{\beta}-e_{\alpha}}N_{Y_{\alpha},Y_{\beta}}+R_{\alpha,\beta}^{e_{\alpha},e_{\beta}}\]
\end{proposition}
\begin{proof}
Again, the four cases are similar, so we just prove the first. Perform reduction first, say, on $B_{\alpha}$ until its charge is equal to zero.
This gives
\[ N_{\alpha,\beta}=t_1^{-e_{\alpha}}N_{Y_{\alpha},\beta}+t_1t_2T_{1,\alpha}^{\ast}+t_1^{-e_{\alpha}}\cdot \left(\begin{array}{cc}
    -1-\dots-t_2^{-e_{\alpha}-1} & \quad\textrm{if}\quad e_{\alpha} < 0  \\
    t_2^{-1}+\dots+t_2^{-e_{\alpha}} & \quad\textrm{if}\quad e_{\alpha} > 0
\end{array} \right) \cdot(V_{0,\beta}-V_{1,\beta})\]
where $Y_{\alpha}$ is the Young diagram part of $B_{\alpha}$, and \[T_{1,\alpha}=L_{1,\alpha}+t_1^{-1}L_{1,\overline{\alpha}}+\dots\] is the space corresponding to the first D3 brane in the triangle part of $B_{\alpha}$. Second, perform reduction on $B_{\beta}$ again until its charge is equal to zero. Note that at this step the charge of $\alpha$ is already zero, so we obtain
\[ 
\begin{multlined}
N_{\alpha,\beta}=t_1^{e_{\beta}-e_{\alpha}}N_{Y_{\alpha},Y_\beta}+t_1t_2T_{1,\alpha}^{\ast}+t_1^{-e_{\alpha}}T_{0,\beta} \\+t_1^{-e_{\alpha}}\cdot \left(\begin{array}{cc}
    -1-\dots-t_2^{-e_{\alpha}-1} & \quad\textrm{if}\quad e_{\alpha} < 0  \\
    t_2^{-1}+\dots+t_2^{-e_{\alpha}} & \quad\textrm{if}\quad e_{\alpha} > 0
\end{array} \right) \cdot(V_{0,\beta}-V_{1,\beta})
\end{multlined}
\]
where 
\[T_{0,\beta}=L_{0,\beta}+t_1^{-1}L_{0,\overline{\beta}}+\dots\] is the space corresponding to the zeroth D3 brane in the triangle part of $B_{\beta}$. The claim then follows by observing that $T_{e_{\beta}-e_{\alpha}}$ is equal to
\[
t_1t_2T_{1,\alpha}^{\ast}+t_1^{-e_{\alpha}}T_{0,\beta} +t_1^{-e_{\alpha}} \cdot \left(\begin{array}{cc}
    -1-\dots-t_2^{-e_{\alpha}-1} & \quad\textrm{if}\quad e_{\alpha} < 0  \\
    t_2^{-1}+\dots+t_2^{-e_{\alpha}} & \quad\textrm{if}\quad e_{\alpha} > 0
\end{array} \right)\cdot (V_{0,\beta}-V_{1,\beta})\]
\end{proof}

For a block $s$ sitting at the $i$th row and $j$th column of a Young diagram $Y$ the leg and arm lengths respectively are defined as
\[ l_Y(s)=Y_i-j \quad\textrm{and}\quad a_Y(s)=Y_j'-i\]
where $Y'$ is the transpose of $Y$.
\begin{proposition} 
\label{prop:Nalphabetam1}
If $m=1$,
\[
\begin{multlined}
N_{\alpha,\beta}(t_1,t_2)= \\ t_1^{e_{\beta}-e_{\alpha}}\left\{\sum_{s \in Y_{\alpha}} \left(t_1^{-l_{Y_{\beta}}(s)}t_2^{a_{Y_{\alpha}}(s)+1}\right) + \sum_{t \in Y_{\beta}}\left(t_1^{l_{Y_{\alpha}}(t)+1}t_2^{-a_{Y_{\beta}}(t)}\right)\right\}+\sum_{(s_1,s_2) \in R_{\alpha,\beta}^{e_{\alpha},e_{\beta}}} t_1^{s_1}t_2^{s_2}
\end{multlined}
\]
\end{proposition}
\begin{proof}
As the charge of both $Y_{\alpha}$ and $Y_{\beta}$ is zero, the formula given in \cite[Theorem~2.11]{nakajima2005instanton} can be applied to express $N_{Y_{\alpha},Y_\beta}$. Combining this with Proposition~\ref{prop:Nalphabetared} we obtain the claim. 
\end{proof}

\subsubsection{General case}


Recall that a Maya diagram $M$ has blocks 
\[k=\ldots,-3/2, -1/2,1/2,3/2,\ldots,\] and each block is an $n\times m$ matrix, where $n$ is the number of D5 branes and $m$ is the number of NS5 branes. Denote the $(E_i,F_j)$ entry of the block $k$ by $M^k_{i,j}$.

\begin{definition}
For an entry $(E_i,F_j)$ in the block $k$ of a Maya diagram define $s_{k,i,j}$ to be the sum of all the entries left of this entry (including this entry). That is,
\[
s_{k,i,j}=\sum_{l=-\infty}^{k-1} \sum_{a=1}^m M^l_{i,a} + \sum_{a=1}^j M^k_{i,a}.
\]
This number is well defined because the Maya diagram is `eventually 0' to the left.
\end{definition}

\begin{definition} A {\em 01-pair} for a Maya diagram is 
\begin{itemize} 
\item an entry 1 in block $k_1$ at position $(E_{i_1},F_j)$, and
\item an entry 0 in block $k_0$ at position $(E_{i_0},F_j)$,
\end{itemize}
if either $k_0>k_1$ or ($k_0=k_1$ and $i_0<i_1$).
\end{definition}

For a 01-pair let $s_1$ be $s_{k,i,j}$ corresponding to the ``1'' in the pair and let $k_1$ be the block number of the ``1'' in the pair.  Similarly, let $s_0$ be $s_{k,i,j}$ corresponding to the ``0'' in the pair and let $k_0$ be the block number of the ``0'' in the pair. 

\begin{theorem}
\label{thm:sum01pairs}
    The Maya diagram $M$ represents a torus fixed point in the bow variety. At this point the tangent space to the bow variety, in torus equivariant K-theory is equal to 
\begin{equation} \label{eq:TangentOfMaya}
\sum_{\text{01-pairs}} \left( 
\frac{\u_{i_0}}{\u_{i_1}} t_1^{s_1-s_0+m(k_0-k_1)} t_2^{s_1-s_0}
+
\frac{\u_{i_1}}{\u_{i_0}} t_1^{1-s_1+s_0-m(k_0-k_1)} t_2^{1-s_1+s_0}
\right)
\end{equation}
\end{theorem}

\begin{proof} {The proof of \cite[Theorem~3.2]{foster2023tangent} generalizes from the type A case to the type $\hat A$ case. For this, note that a tie diagram of type $\hat A$ determines a tie diagram of type $A$ if we cut off the irrelevant parts outside of the ties. This is equivalent with forgetting the irrelevant part of the Maya diagram; the process is compatible with the dimension of the spaces $V_j$ and hence with those of the tangent spaces due to the construction of the Maya diagram. By the block periodicity in the affine case we have to add $m(k_0-k_1)$ to the $t_1$-exponent $s_1-s_0$. Note also that the four elementary diagram pieces in \cite[Page~11]{foster2023tangent} can be endowed with our $(\mathbb{C}^{\ast})^2$-action in a compatible manner and this action carries through the whole proof.}
\end{proof}

\begin{example}
Consider the first Maya diagram in Example~\ref{ex:2Maya}. It encodes a torus fixed point on a bow variety of dimension 36.  Correspondingly, expression \eqref{eq:TangentOfMaya} for this Maya diagram has 36 terms, corresponding to the 18 possible 01-pairs. One of the 01-pairs is 
\begin{itemize}
\item 1 in block -1/2 at position $(E_3,F_1)$ (yielding $i_1=3, s_1=1, k_1=-1/2$), and
\item 0 in block 3/2 at position $(E_2,F_1)$ (yielding $i_0=2, s_0=1, k_0=3/2$).
\end{itemize}
The two-term contribution of this 01-pair to \eqref{eq:TangentOfMaya} is
\[ 
\frac{\u_2}{\u_3} t_1^{4} + \frac{\u_3}{\u_2} t_1^{-3}t_2. 
\]
\end{example}

For $a,b \in \ZZ$, we set
\[ \delta_{a,b}^{(m)} \coloneqq \begin{cases}
1 & \textrm{if } a \equiv b \; (\mathrm{mod}\, m) \\
0 & \textrm{if } a \not\equiv b \; (\mathrm{mod}\, m).
\end{cases}
\]
Take two Young diagrams $Y_{\alpha}$ and $Y_{\beta}$. For a block $s$ in any of these two diagrams we define its \emph{relative hook length}
\[ \u_{Y_{\alpha},Y_{\beta}}(s) \coloneqq l_{Y_\beta}(s) + a_{Y_\alpha}(s) + 1.\]
Using these notations we can now give a refinement of Proposition~\ref{prop:Nalphabetam1} for arbitrary $m$. It is worth to compare this result with \cite[Theorem~3.4]{nakajima2005instanton}.
\begin{theorem}
\label{thm:tangentcharm}
Let $M=(M_1,\dots,M_n)$ be a Maya diagram corresponding to a T-fixed point of a bow variety with $(B_1,\dots,B_n)$ be the corresponding tuple of extended Young diagrams. 
At this point the tangent space to the bow variety in torus equivariant K-theory is equal to
\[ \sum_{\alpha,\beta=1}^n N_{\alpha,\beta}(t_1,t_2) \u_{\beta}\u_{\alpha}^{-1} \]
where
\[
\begin{multlined}
N_{\alpha,\beta}(t_1,t_2)= \\ t_1^{e_{\beta}-e_{\alpha}}\left\{\sum_{s \in Y_{\alpha}} \left(t_1^{-l_{Y_{\beta}}(s)}t_2^{a_{Y_{\alpha}}(s)+1}\right)\delta^{(m)}_{\u_{Y_{\beta},Y_{\alpha}}(s),e_{\alpha}-e_{\beta}}\right. \\ +\left. \sum_{t \in Y_{\beta}}\left(t_1^{l_{Y_{\alpha}}(t)+1}t_2^{-a_{Y_{\beta}}(t)}\right)\delta^{(m)}_{\u_{Y_{\alpha},Y_{\beta}}(s),e_{\beta}-e_{\alpha}}\right\}+\sum_{(s_1,s_2) \in R^{e_{\alpha},e_{\beta}}_{\alpha,\beta}} t_1^{s_1}t_2^{s_2}\delta_{s_1,s_2}^{(m)}.
\end{multlined}
\]
and $R_{\alpha,\beta}^{e_{\alpha},e_{\beta}}$ is defined in \eqref{eq:Rdef}.
\end{theorem}
\begin{proof}
By forgetting the $\ZZ/m\ZZ$ equivariant structure we can consider the Maya diagram $M$ as a Maya diagram $M'$ of a fixed point of a bow variety with a single NS5 brane. In this correspondence, say, in the positive direction the $j$th column of block $k/2$ of $M$ becomes the single column of block $m(k-1)+j/2$ in $M'$. 

By Theorem~\ref{thm:sum01pairs} the $T$-module corresponding to $M$ is the summand of the $T$-module corresponding to $M'$ comprising those monomials $t_1^at_2^b$ for which $a \equiv b \; (\mathrm{mod}\, m)$.
The monomials in the claim are exactly these ones.

\end{proof}

\subsection{The effect of D5 swap}
\label{sec:D5again}

Let us consider again the combinatorial operation called D5 swap depicted in \eqref{eq:D5swap}; for simplicity we will restrict our attention to separated brane diagrams. The associated bow varieties will be called $\mathcal{M}=\mathcal{M}(d,e,f)$ and $\mathcal{M}'=\mathcal{M}(d,e',f)$ respectively. The vector $e'$ is obtained from $e$ by swapping the $j$'th and $j+1$'st components, where $E_j$ and $E_{j+1}$ are the two swapped D5 branes.

Let $x\in \mathcal{M}(d,e,f)$ and $x'\in\mathcal{M}(d,e',f)$ be torus fixed points in the two varieties, corresponding to each other under the bijection illustrated in Figure \ref{fig:D5swap}. 
The Maya diagram of $x'$ is obtained from the Maya diagram of $x$ by swapping the $j$'th and $j+1$'st rows.

\begin{theorem} 
\label{thm:D5swap_tangent}
    Let $\delta e=e_{j+1}-e_j$. We have
\begin{equation}\label{eq:D5swap_tangent}
T_x\mathcal{M} - T_{x'}\mathcal{M}'|_{\u_j\leftrightarrow \u_{j+1}}=
W+ t_1t_2 W^*
\qquad
\in K_T(\pt),
\end{equation}
where 
\[
W=\begin{cases}
-\frac{\u_{j+1}}{\u_{j}} \sum_{i=1}^{\delta e} (t_1t_2)^i 
&  \text{if } \delta e\geq 0 
\\
\phantom{-}\frac{\u_{j+1}}{\u_{j}} \sum_{i=\delta e+1}^{0} (t_1t_2)^i 
& \text{if } \delta e<0
\end{cases}
\]
\end{theorem}

\begin{proof}
 Let $M$ be the Maya diagram of $x$, and $M'$ the Maya diagram of $x'$; they are the same except the $j$'th and $j+1$'th columns are swapped. Consider the descriptions of $T_x\mathcal{M}$ and $T_{x'}\mathcal{M}'$ as sums of expressions parameterized by 01-pairs in $M$ and $M'$, cf.~Theorem~\ref{thm:sum01pairs}. Most of the 01-pairs of $M$ and $M'$ are in obvious bijection, and the corresponding terms are the same for $T_x\mathcal{M}$ and $T_{x'}\mathcal{M}'|_{\u_j\leftrightarrow \u_{j+1}}$. The only exceptions are the 01-pairs with 0 and 1 in rows $j$ and $j+1$ above each other:
 \begin{enumerate}
 \item[(a)]  entry 1 in row $E_j$ and entry 0 in row $E_{j+1}$ in the same column,
 \item[(b)] entry 0 in row $E_j$ and entry 1 in row $E_{j+1}$ in the same column. 
\end{enumerate}
An occurrence of (a) contributes two terms to $T_x\mathcal{M}$, and an occurrence of (b) contributes two terms to $T_{x'}\mathcal{M}'|_{\u_j\leftrightarrow \u_{j+1}}$. The contributions of all other 01-pairs to $T_x\mathcal{M} - T_{x'}\mathcal{M}'|_{\u_j\leftrightarrow \u_{j+1}}$ cancel, hence we disregard them.

To calculate the contributions of these vertical 01-pairs, let us consider rows $j$ and $j+1$ only, as in the picture
\[
\begin{tikzpicture}
 \draw (0,-.2) -- (10,-.2);
 \draw (0,.6) -- (10,.6);
  \draw (1,0) node {$1$};
  \draw (1,.4) node {$0$};
  \draw (2,0) node {$1$};
  \draw (2,.4) node {$0$};
   \draw (3,0) node {$0$};
  \draw (3,.4) node {$1$};
   \draw (4,0) node {$0$};
  \draw (4,.4) node {$1$};
   \draw (5,0) node {$0$};
  \draw (5,.4) node {$1$};
   \draw (6,0) node {$0$};
  \draw (6,.4) node {$1$};
   \draw (7,0) node {$1$};
  \draw (7,.4) node {$0$};
   \draw (8,0) node {$0$};
  \draw (8,.4) node {$1$};
   \draw (9,0) node {$0$};
  \draw (9,.4) node {$1$};

 \draw (0.8,-.2) -- (.8,.6); \draw (1.2,-.2) -- (1.2,.6);
  \draw (1.8,-.2) -- (1.8,.6); \draw (2.2,-.2) -- (2.2,.6);
   \draw (2.8,-.2) -- (2.8,.6); \draw (3.2,-.2) -- (3.2,.6);
    \draw (3.8,-.2) -- (3.8,.6); \draw (4.2,-.2) -- (4.2,.6);
     \draw (4.8,-.2) -- (4.8,.6); \draw (5.2,-.2) -- (5.2,.6);
      \draw (5.8,-.2) -- (5.8,.6); \draw (6.2,-.2) -- (6.2,.6);
       \draw (6.8,-.2) -- (6.8,.6); \draw (7.2,-.2) -- (7.2,.6);
        \draw (7.8,-.2) -- (7.8,.6); \draw (8.2,-.2) -- (8.2,.6);
         \draw (8.8,-.2) -- (8.8,.6); \draw (9.2,-.2) -- (9.2,.6);
  
\draw ( 12,0) node {$E_{j+1}$};
\draw ( 12,.4) node {$E_j$};
 \draw (1,.9) node {(b)};
 \draw (2,.9) node {(b)};
  \draw (3,.9) node {(a)};
   \draw (4,.9) node {(a)};
    \draw (5,.9) node {(a)};
     \draw (6,.9) node {(a)};
      \draw (7,.9) node {(b)};
       \draw (8,.9) node {(a)};
        \draw (9,.9) node {(a)};
 \draw (-1,-.5) node {$\Delta=$};
  \draw (0,-.5) node {$\ldots 0\ldots$};
 \draw (1.15,-.5) node {$-1 \ldots$};
  \draw (2.15,-.5) node {$-2 \ldots$};
   \draw (3.15,-.5) node {$-1\ldots$};
    \draw (4.25,-.5) node {$0\ldots$};
     \draw (5.25,-.5) node {$1\ldots$};
      \draw (6.25,-.5) node {$2\ldots$};
       \draw (7.25,-.5) node {$1\ldots$};
        \draw (8.25,-.5) node {$2\ldots$};
         \draw (9.25,-.5) node {$3\ldots$};
          \draw (10.25,-.5) node {$3 \ldots$};
         
\end{tikzpicture}
\]
where in between the indicated columns all other columns are either $\begin{bsmallmatrix}0 \\0 \end{bsmallmatrix}$ or $\begin{bsmallmatrix}1 \\1 \end{bsmallmatrix}$. For each column let $\Delta$ denote the number of type (a) columns minus the number of type (b) columns left of (and including) this column. 

We claim that the total contribution of the vertical 01-pairs from column $-\infty$ to column $k$ is $W+ t_1t_2 W^*$ where
\[
W=\begin{cases}
-\frac{\u_{j+1}}{\u_{j}} \sum_{i=1}^{\Delta} (t_1t_2)^i 
&  \text{if } \Delta \geq 0 
\\
\phantom{-}\frac{\u_{j+1}}{\u_{j}} \sum_{i=\Delta+1}^{0} (t_1t_2)^i 
& \text{if } \Delta <0.
\end{cases}
\]
For $k\ll 0$ this statement is vacuously true, and the $k\to k+1$ induction step follows (in the four possible $\begin{bsmallmatrix}0 \\0 \end{bsmallmatrix}$, $\begin{bsmallmatrix}1 \\0 \end{bsmallmatrix}$, $\begin{bsmallmatrix}0 \\1 \end{bsmallmatrix}$, $\begin{bsmallmatrix}1 \\1 \end{bsmallmatrix}$ cases) from the formula in Theorem~\ref{thm:sum01pairs}. 

It remains to prove that $\Delta$ in a column far right is equal to $e_{j+1}-e_j$. For this let $l_1, l_2, l_3$ denote the number of  
$\begin{bsmallmatrix}1 \\1 \end{bsmallmatrix}$, 
$\begin{bsmallmatrix}1 \\0 \end{bsmallmatrix}$, 
$\begin{bsmallmatrix}0 \\1 \end{bsmallmatrix}$ occurrences in rows $j$ and $j+1$ {\em in negative blocks}. Let $r_1, r_2, r_3$ denote the number of
$\begin{bsmallmatrix}1 \\0 \end{bsmallmatrix}$, 
$\begin{bsmallmatrix}0 \\1 \end{bsmallmatrix}$, 
$\begin{bsmallmatrix}0 \\0 \end{bsmallmatrix}$ occurrences in rows $j$ and $j+1$ {\em in positive blocks}. Then we have that $\Delta$ on the far right is equal to 
\[
(l_2+r_1)-(l_3+r_2)
=
\left( (r_1+r_3)-(l_1+l_3) \right)
-
\left( (r_2+r_3)-(l_1+l_2) \right)
=
e_{j+1}-e_j,
\]
which completes the proof.
\end{proof}

A remarkable consequence of Theorem \ref{thm:D5swap_tangent} is that the difference \eqref{eq:D5swap_tangent} only depends on the ambient variety $\mathcal{M}$, not on the fixed point $x \in \mathcal{M}$. There {\em could} be a simple geometric explanation for this phenomenon: specifically, if $\mathcal{M}'$ were the total space of a vector bundle over $\mathcal{M}$ (with a $T$ action), or vice versa. Indeed, in many examples this is the case—one of $\mathcal{M}$ or $\mathcal{M}'$ is a vector bundle over the other. However, this is not always true; a counterexample is provided in \cite{ji2023bow}, where one of the varieties is a point, and the other has odd cohomology.

Although our desire for $\mathcal{M}'$ to be a vector bundle over $\mathcal{M}$ (or vice versa) is not strictly accurate, Theorem \ref{thm:D5swap_tangent} demonstrates that near the torus fixed points, they resemble a situation where one is a vector bundle over the other.

\section{Refined generating series}
\label{sec:refinedgenseries}

Our aim in this section is to refine Theorem \ref{thm:eulerchargen} in various directions---using the formulas on the torus characters at fixed point proved in Section \ref{sec:equivK}. Below we will freely shift between the notations $\mathcal{M} = \mathcal{M}(d,e,f) = \mathcal{M}(\underline{d})$.

\subsection{Cells from the torus action}
\label{subsec:celldecomposition}


Let $\mathcal{M}_{\mathrm{0}}$ be the affine algebro-geometric quotient $\mu^{-1}(0)/\!\!/ \mathcal{G}$. This is the bow variety corresponding to the stability parameter $\nu^{\RR}=0$ (recall that we set $\nu^{\CC}=0$ already). 
The map
\[ \pi: \mathcal{M} \to \mathcal{M}_{\mathrm{0}}\]
is semismall onto its image \cite[Proposition~4.5]{nakajima2017cherkis}, in particular, it is proper. As the torus action on $\mathcal{M}$ is Hamiltonian, $\pi$ is $T$-equivariant.

By \cite[(4.3)]{nakajima2017cherkis} there exists a decomposition
\[ \mathcal{M}_0 = \bigsqcup_{\underline{d}',\lambda} \mathcal{M}^s(\underline{d}') \times S^{\lambda}(\CC^2 \setminus \{0\} / (\ZZ/m\ZZ) ) \]
where $m$ is the number of NS5 branes, $\mathcal{M}^s(\underline{d}')$ is the stable part of a bow variety associated with the same brane diagram and a dimension vector $\underline{d}' \leq \underline{d}$, $S^{\lambda}$ is the stratum of the symmetric product $S^{|\lambda|}(\CC^2 \setminus \{0\} / (\ZZ/m\ZZ) )$ consisting of configurations whose multiplicities are given by the partition $\lambda$. The difference $|\underline{d}|-|\underline{d}'|-|\lambda|$ is the `multiplicity' at the origin $0 \in \CC^2 / (\ZZ/m\ZZ)$. Consider the component in $\mathcal{M}_0$ such that $\underline{d}'=0$ and $|\lambda|=0$; we denote it by $[0]$. 
The subvariety $\pi^{-1}([0]) \subset \mathcal{M}$ can be thought of as the analogue of the punctual Quot scheme\footnote{Some sources \cite{kamnitzer2022symplectic, ji2023bow} call this the \emph{core} of the bow variety. This should not be confused with the core of a Maya diagram.}.

For a generic one-parameter subgroup $\gamma: \mathbb{C}^{\ast} \to T$ the fixed point set coincides with that of $T$:
\[ \mathcal{M}^{\gamma(\mathbb{C}^{\ast})}=\mathcal{M}^T.\]
Then for each fixed point $x \in \mathcal{M}^T$ one can consider $(+)$, respectively $(-)$, attracting sets:
\[ S_x = \left\{y \in M\;|\; \lim_{t \to 0} \gamma(t)\cdot y= x\right\}, \]
\[ U_x = \left\{y \in M\;|\; \lim_{t \to \infty} \gamma(t)\cdot y= x\right\}.\]
These are affine spaces by \cite{bialynicki1973some}. Moreover, there exists an order on the fixed points so that $\bigcup_{z \leq x} S_z$ (resp. $\bigcup_{z \leq x} U_z$) is closed in $\bigcup_y S_y$ (resp. $\bigcup_y U_y$) for each $x$.

We introduce generating series for the virtual Betti numbers arising from the above $+/-$ cells:
\[ Z^{-}(q,t)=Z^{-}_{e,f}(q,t)=\sum_{d \geq 0} P^-_t(\mathcal{M}(d,e,f))q^d \coloneqq \sum_{d \geq 0,x} P_t(U_x)q^d \]
\[ Z^{+}(q,t)=Z^{+}_{e,f}(q,t)=\sum_{d \geq 0} P^+_t(\mathcal{M}(d,e,f))q^d \coloneqq\sum_{d \geq 0,x} P_t(S_x)q^d \]
where $x$ runs through fixed points of $\mathcal{M}(d,e,f)$ (in notation we do not indicate the choice of one-parameter subgroup $\gamma$, it should be clear from the context).

The torus action on the bow variety $\mathcal{M}$  is called \emph{circle compact} if 
\begin{equation} 
\label{eq:cellscover}
\bigcup_x S_x = \mathcal{M},\quad  \bigcup_x U_x=\pi^{-1}([0])\end{equation}
hold. 
\begin{proposition} 
\label{prop:bbbowvar}
Assume that $\mathcal{M}$ satisfies \eqref{eq:cellscover} for a one-parameter subgroup $\gamma$. Then
\begin{enumerate}
\item $H_{\mathrm{odd}}(\mathcal{M},\ZZ)=0$ \item $H_{\mathrm{even}}(\mathcal{M},\ZZ)$ is a free abelian group
\item the cycle map from the Chow group \[A_{\ast}(\mathcal{M})\to H_{\mathrm{even}}(\mathcal{M},\ZZ)\] is an isomorphism
\item $\pi^{-1}([0])$ is a homotopy retract of $\mathcal{M}$
\end{enumerate}
\end{proposition}
\begin{proof} This is standard, see e.g. \cite{nakajima2004lectures}.
\end{proof}
\begin{corollary} 
\label{cor:bbbowvarfam}
If $\mathcal{M}(d,e,f)$ satisfies \eqref{eq:cellscover} for a one-parameter subgroup $\gamma$, for all $d\geq 0$, then we have
\begin{enumerate}
\item $Z(q,t)=Z^{-}(q,t)$,
\item $\mathcal{Z}(q)\vert_{\mathbb{L}=t^2}=Z^{+}(q,t)$
\end{enumerate}
where $\mathbb{L}$ is the class of the affine line in $K_0(Var)$.
\end{corollary}

\begin{remark}
In \cite[Section~4.3]{ji2023bow} a finite type A example is provided (which can naturally be realized as affine type A) where $\bigcup U_x \varsubsetneq \pi^{-1}([0])$.
\end{remark}

\begin{proposition} \label{prop:MaxCellCircleCompact}
Suppose that $\mathcal{M}$ has a maximal dimensional $(+)$-cell, or, equivalently, a zero dimensional $(-)$-cell. Then the action on $\mathcal{M}$ is circle compact and \eqref{eq:cellscover} holds.
\end{proposition}
\begin{proof}
By the assumption there is a dense orbit of the torus action on $\mathcal{M}_{0}$ such that all points in this orbit flow to the image, say 0, of the fixed point of the above cells. Therefore, $[0]=0$. As the closure of this orbit is $\mathcal{M}_{0}$, we obtain that $\mathcal{M}_{0}$ is a cone for the $\CC^{\ast}$-action. As a consequence, $\lim_{t \to 0} t.x$ exist for all $x \in \mathcal{M}$.
\end{proof}

\medskip

\noindent{\em The choice of the one-parameter subgroup.} In Sections \ref{subsection:Poincarem1}, \ref{sec:Poincare2}, and \ref{sec:Poincare3} we will prove formulas for the $Z^+, Z^-$ series, using a specific choice of one-parameter subgroup $\gamma$. Namely, let
\[
\gamma:\C^* \to  T=\C^*_{t_1} \times \C^*_{t_2}\times \left( \C^* \right)^{\text{D5 branes}}, 
\qquad\qquad
t \mapsto  (t^{m_1},t^{m_2},t^{r_1},\dots,t^{r_n}).
\]
For generic weights $m_1, m_2, r_{\alpha}$ the Zariski closure of $\gamma(\mathbb{C}^{\ast})$ is the whole $T$, and the fixed point set of $\gamma(\mathbb{C}^{\ast})$ coincides with that of $T$.
Our choice for $\gamma$ will be
\begin{equation}\label{eq:subgroup_choice} 
m_2 \gg r_1 > \dots > r_n \gg m_1 > 0. 
\end{equation}
This choice of one-parameter subgroup is consistent with the analogous choice for quiver varieties in \cite{nakajima2005instanton}.

\subsection{Zastavas ($m=1$)}
\label{subsection:Poincarem1}

In this section we assume that the brane diagram contains only one NS5 brane. 

\begin{theorem} 
\label{thm:poincare_zastava}
Let $m=1$. Then for $\mathcal{M}=\mathcal{M}(d,e,f)$ we have 
\[P^{-}_t(\mathcal{M})=\sum_{(B_1,\dots,B_n)} \prod_{\beta=1}^n t^{2\left(n|Y_{\beta}|-\beta l(Y_\beta) + \sum_{1 \leq \alpha < \beta\leq n} \binom{|e_{\beta}-e_{\alpha}-1|}{2}\right)} 
\]
where the summation runs over $n$-tuples of extended Young diagrams $(B_1,\dots,B_n)$ with $d$ fixed, $(Y_1,\dots,Y_n)$ are their classical parts and $l(Y_i)$ is the number of columns of $Y_i$.
\end{theorem}
\begin{remark}
\begin{enumerate}
    \item In the exponent of $t$ the first two terms depend only on the Young diagram (quotient) part of the $B_{\beta}$, while the terms in the last summation depend only on the margin vector $e$. 
\item Theorem~3.8 of \cite{nakajima2004lectures} is the special case of our
Theorem~\ref{thm:poincare_zastava} corresponding to $e_i=0$ for all $i$. In this case the summation term in the exponent of $t$ vanishes. 
\end{enumerate}
\end{remark}

Note that since $m=1$ the margin vector $e$ uniquely determines the contingency table $c$ as $c_{i1}=e_1$, $1 \leq i \leq n$. Moreover, $f=f_1=\sum_i e_i$. As a consequence, we have the following.
\begin{corollary} \label{cor:Zminus,m=1}
For $m=1$ and $e \in \ZZ^n$ we have
\[Z^-(q,t) = 
q^A
\cdot 
t^{2B}
\cdot \prod_{i=1}^n \prod_{l \geq 1}^{\infty} \frac{1}{1-t^{2(nl-i)}q^l}\]
where 
\[
A=\sum_{\alpha=1}^n \frac{e_\alpha(e_\alpha-1)}{2}, \qquad\qquad
B=\sum_{1\leq \alpha<\beta\leq n} \binom{|e_\beta-e_\alpha-1|}{2}.
\]
\end{corollary}

\begin{remark}
    The $B=0$ condition implies circle compactness of the varieties $\mathcal{M}(d,e,(\sum e_i))$. 
    The formula above shows that $B=0$ is equivalent to $e_1\leq e_2\leq \ldots \leq e_n\leq e_1+2$, that is, $e$ being a 2-bounded non-decreasing sequence, cf. Theorem~\ref{thm:cbbcelldec_cover}.
\end{remark}

\begin{example} \label{ex:e4f1}
    For $e=(-3,2,-3,4)$ we obtain
    \begin{multline*}
    Z^-(q,t)=
    q^{19}t^{102}+q^{20}(t^{102}+t^{104}+t^{106}+t^{108})+
    q^{21}(t^{102} + t^{104} + \\2t^{106} + 2t^{108} +3t^{110} + 2t^{112} + 2t^{114} + t^{116})+O(q^{21})=
    q^{19}t^{102}\sum_{s=0}^\infty P^-_s(t)q^s.
    \end{multline*}
Computer evidence suggests that $P^-_s(t)$ stabilizes 
     \[
     P^-_s(t)\to \prod_{r=1}^{\infty} \frac{1}{(1-t^{2r})}
     \qquad\qquad 
     \text{ as } s\to\infty
     \]
     in the following sense: $P^-_s(t)$ is equal to $\prod_{r=1}^{\infty} \frac{1}{(1-t^{2r})}$ up to order $t^{2s}$. 
\end{example}

     
\begin{proof}[{Proof of Theorem~\ref{thm:poincare_zastava}}]

Let $x \in \mathcal{M}$ be a fixed point of the $T$-action. Then $T_x\mathcal{M}$ has a $T$-module
structure, and has an induced $\mathbb{C}^{\ast}$-module structure via $\gamma$.
By our choice \eqref{eq:subgroup_choice} of $\gamma$ the negative weight space for the $\gamma$-action is the direct sum of weight spaces for the $T$-action such that one of the following hold:
\begin{itemize}
\item the weight of $t_2$ is negative,
\item the weight of $t_2$ is zero and the weight of $u_1$ is negative,
\item the weight of $t_2$, $u_1$ are zero and the weight of $u_2$ is negative,

\dots

\item the weight of $t_2$, $u_1$, $u_2$, \dots , $u_{n-1}$ are zero and the weight of $u_n$ is negative,
\item the weight of $t_2$, $u_1$, $u_2$, \dots, $u_n$ are zero and weight of $t_1$ is negative.
\end{itemize}
(The reader may find it instructive to verify that the number of such negative terms in the five expressions in \eqref{03-14-example} are $1$, $2$, $1$, $2$, $0$, respectively. In particular, for the last fixed point the negative weight space is 0, cf.~Proposition~\ref{prop:MaxCellCircleCompact}.)

We calculate the sum of dimensions of weight spaces with the above
condition in each summand of 
Theorem~\ref{thm:tangentcharm}
separately, and then sum up the contribution from
each summand. In the summand $\alpha = \beta$, the contribution is
\[|Y_{\alpha}| - l(Y_{\alpha})\]
where $l(Y_{\alpha})$ is the length (number of parts) of $Y_{\alpha}$. 
If $\alpha < \beta$, the above condition is equivalent to that the weight of $t_2$ is nonpositive. 
Hence the contribution can be written as
\[|Y_{\beta}|+ \binom{e_{\beta}-e_{\alpha}-1}{2}\]
where the binomial term is understood to be zero if $e_{\beta}-e_{\alpha}-1 \leq 0$.
If $\alpha > \beta$, the above condition is equivalent to that the weight of $t_2$ is negative. 
The contribution is hence
\[|Y_{\beta}|- l(Y_{\beta})+\binom{e_{\beta}-e_{\alpha}-1}{2}.\]
Adding up all terms we get
\[ \sum_{\beta=1}^n n|Y_{\beta}|-(n-\beta+1)l(Y_{\beta}) + \sum_{1 \leq \alpha < \beta} \binom{e_{\beta}-e_{\alpha}-1}{2}+\binom{e_{\alpha}-e_{\beta}-1}{2}. \]
\end{proof}

The $(+)$-series can be obtained similarly (the details are left to the reader).
\begin{theorem} 
\label{thm:m1motivic}

We have
\[
Z^+(q,t) = q^A \cdot t^{2B} \cdot \prod_{i=1}^n \prod_{l \geq 1}^{\infty} \frac{1}{1-t^{2(nl+i)}q^l} 
\]
where 
\[
A=\sum_{\alpha=1}^n \frac{e_\alpha(e_\alpha-1)}{2}, \qquad
B=\sum_{1\leq \alpha<\beta\leq n} F(e_\alpha-e_\beta),
\qquad
F(l)=\begin{cases}
(l+1)l/2 & \textrm{if } l\geq 0
\\
(l-1)l/2-1 & \textrm{if } l<0.
\end{cases}
\]
\end{theorem}

\subsection{The case $n=1$} \label{sec:Poincare2}
In another direction we deduce explicit formulas when $m$ is arbitrary but there is only one D5 brane.  Again, the margin vectors in this case completely determine the core via $c_{1j}=f_j$, $1 \leq j \leq m$.
\begin{proposition} 
\label{prop:n1marbgenfn}
For $n=1$ and $f \in \ZZ^m$ we have
\[Z^-(q,t) = 
q^A
\cdot \prod_{l \geq 1}^{\infty} \frac{1}{(1-t^{2l-2}q^l)(1-t^{2l}q^l)^{m-1}} \]
and
\[Z^+(q,t) = 
q^{A} \cdot \prod_{l \geq 1}^{\infty} \frac{1}{(1-t^{2l+2}q^l)(1-t^{2l}q^l)^{m-1}} \]
where 
\[
A=\sum_{i=1}^m \frac{f_i(f_i-1)}{2}. \]
\end{proposition}

\begin{proof} As $n=1$, in Theorem~\ref{thm:tangentcharm} we only have $\alpha=\beta=1$. Hence, the exponent of $t$ in the first term is zero. Moreover, the last term vanishes. 
The direct analog of the computation in \cite[Lemma~4.7]{fujii2017combinatorial} gives the above generating series (see also \cite[Theorem~1.3]{gyenge2018euler}). 
\end{proof}

\subsection{The case $n>1$ and $m>1$} \label{sec:Poincare3}

\subsubsection{Stabilization}

First, let us illustrate the $n,m>1$ case by an example.
\begin{example} 
\label{exe2f2}
     For $e=f=(0,0)$ we have
    \[Z(q,t)=\sum_{s=0}^\infty P_s(t)q^s,
    \qquad
    Z(q)=\sum_{s=0}^\infty \left( \mathbb{L}^{4s} P_s(\mathbb{L}^{-1/2}) \right) q^s
    \]
with
\[
\begin{array}{rl}
    P_0(t)= & 1 \\
    P_1(t)= & 1+2t^2+t^4 \\
    P_2(t)= & 1+2t^2+5t^4+5t^6+3t^8 \\
    P_3(t)= & 1+2t^2+5t^4+10t^6+13t^8+12t^{10}+5t^{12} \\
    P_4(t)= & 1+2t^2+5t^4+10t^6+20t^8+28t^{10}+33t^{12}+24t^{14}+10t^{16}, \ldots
\end{array}
  \]  
and $P_s(t)$ stabilizes 
     \[
     P_s(t)\to \prod_{r=1}^{\infty} \frac{1}{(1-t^{2r})^2}
     \qquad\qquad 
     \text{ as } s\to\infty
     \]
     in the following sense: $\deg(P_s(t))=4s$, and $P_s(t)$ agrees with \[\prod_{r=1}^{\infty} \frac{1}{(1-t^{2r})^2}=
1+2t^2+5t^4+10t^6+20t^8+36t^{10}+65t^{12}+110t^{14}+\ldots\]
     up to degree $2s$. Remarkably, the top (degree $4s$) coefficient of $P_s(t)$ is the number of partitions of $2s-2$ such that all odd parts are distinct. 
\end{example}

The stabilization phenomenon observed in Examples \ref{ex:e4f1} and \ref{exe2f2} also appears in other examples:

\begin{conjecture} 
For $e\in \Z^n, f\in \Z^m$ let $Z(q,t)=\sum_{s=0}^\infty P_s(t)q^s$. The polynomial $P_s(t)$ agrees with the series
$\prod_{r=1}^{\infty} {1}/{(1-t^{2r})^m}$ up to degree $N_s$ where $N_s\to \infty$.
\end{conjecture}

\subsubsection{Series associated with a fixed core}

Theorems~\ref{thm:sum01pairs} and \ref{thm:tangentcharm} can also be used to compute the dimension of any cell. Denote by $Z^c_{\mathrm{quotient}}(q,t)$ the series enumerating the $(-)$-cells corresponding to descendants of a fixed core $c$.
\begin{example} For $e=(0,0)$ and $f=(0,0)$ a few examples of cores and their quotient series are as follows:
\[
\begin{multlined}
Z^{\text{\begin{tiny}\(\begin{bmatrix}0 & 0 \\ 0 & 0\end{bmatrix}\)\end{tiny}}}_{\mathrm{quotient}}(q,t)=
1 + q \left(t^{4} + 2 t^{2} + 1\right) + q^{2} \cdot \left(3 t^{8} + 4 t^{6} + 4 t^{4} + 2 t^{2} + 1\right) \\ + q^{3} \cdot \left(3 t^{12} + 5 t^{10} + 10 t^{8} + 5 t^{6} + 4 t^{4} + t^{2}\right) + O\left(q^{4}\right)
\end{multlined}
\]
\[
\begin{multlined}
Z^{\text{\begin{tiny}\(\begin{bmatrix}1 & -1 \\ -1 & 1\end{bmatrix}\)\end{tiny}}}_{\mathrm{quotient}}(q,t)=
 1 + q \left(2 t^{4} + t^{2} + 1\right) + q^{2} \left(t^{10} + 3 t^{8} + 4 t^{6} + 4 t^{4} + t^{2} + 1\right) \\ + q^{3} \left(t^{14} + 4 t^{12} + 6 t^{10} + 9 t^{8} + 4 t^{6} + 3 t^{4} + t^{2}\right) + O\left(q^{4}\right)
\end{multlined}
\]
\[
\begin{multlined}
Z^{\text{\begin{tiny}\(\begin{bmatrix}-1 & 1 \\ 1 & -1\end{bmatrix}\)\end{tiny}}}_{\mathrm{quotient}}(q,t)= 1 + q \left(t^{6} + t^{4} + t^{2} + 1\right) + q^{2} \cdot \left(3 t^{10} + 3 t^{8} + 3 t^{6} + 3 t^{4} + t^{2} + 1\right) \\ + q^{3} \cdot \left(2 t^{14} + 5 t^{12} + 9 t^{10} + 5 t^{8} + 4 t^{6} + 2 t^{4} + t^{2}\right) + O\left(q^{4}\right)
\end{multlined}
\]
All three series seems to be divisible by
\[ \phi(q,t)=\prod_{l \geq 1}^{\infty} \frac{1}{(1-t^{4l-4}q^l)(1-t^{4l-2}q^l)(1-t^{4l}q^l)} \]
in the sense that they are a product of $\phi(q,t)$ with power series having nonnegative integer coefficients.
Additionally, as expected, they all specialize to $\prod_{l \geq 1}^{\infty} (1-q^l)^{-4}$ at $t=1$.
The factors divided further by the contribution of the core cell are
\[
\frac{
Z^{\text{\begin{tiny}\(\begin{bmatrix}0 & 0 \\ 0 & 0\end{bmatrix}\)\end{tiny}}}_{\mathrm{quotient}}}{\phi(q,t)}= 1 + q t^{2} + q^{2} \left(t^{8} + t^{6}\right) + q^{3} \left(t^{12} + t^{10} + t^{8}\right) + q^{4} \left(t^{16} + 2 t^{14} + 2 t^{12}\right) + O\left(q^{5}\right)
\]
\[
\frac{Z^{\text{\begin{tiny}\(\begin{bmatrix}1 & -1 \\ -1 & 1\end{bmatrix}\)\end{tiny}}}_{\mathrm{quotient}}}{\phi(q,t) q^2 t^6}= 1 + q t^4 + q^{2} \left(t^{10} + t^{6}\right) + q^{3} \left(t^{14} + t^{12} + t^{8}\right) + q^{4} \left(3t^{16} + 2 t^{14}\right) + O\left(q^{5}\right)
\]
\[
\frac{Z^{\text{\begin{tiny}\(\begin{bmatrix}-1 & 1 \\ 1 & -1\end{bmatrix}\)\end{tiny}}}_{\mathrm{quotient}}}{ \phi(q,t)q^2t^4}=1+qt^6+2q^2t^{10}+q^3(t^{14}+2t^{12})+q^4(t^{20}+t^{18}+2t^{16}+t^{14})+ O\left(q^{5}\right)
\]

\end{example}

\begin{corollary} 
\label{cor:noprodnmarb}
The motivic/Poincaré generating series do not in general decompose into a product of core-term and quotient-term when $n>1$ and $m>1$.
\end{corollary}

\begin{conjecture}For $n,m$ and $e,f$ arbitrary, $Z_{\mathrm{quotient}}^c(q,t)$ is divisible in sense above by 
\[ \prod_{l \geq 1}^{\infty} \frac{1}{(1-t^{2(nl-n+1)}q^l)\cdots (1-t^{2(nl-1)}q^l)(1-t^{2nl}q^l)^{m-1}}, \]
\end{conjecture}

It seems to be an interesting combinatorial problem to express the remaining factors from the core and its contingency table.

\subsection{Cell decomposition} \label{sec:CellDecomposition}

We now give a sufficient condition for the Bia\l ynicki-Birula cells to cover the bow variety $\mathcal{M}(d,e,f)$. We continue using the one-parameter subgroup $\gamma$ specified in \eqref{eq:subgroup_choice}.
\begin{theorem}
\label{thm:cbbcelldec_cover}
Assume that $e \in \ZZ^n$ is an {\em $m+1$-bounded non-decreasing sequence}, that is 
\begin{equation}
\label{eq:weakquivcond}
e_1 \leq e_2 \leq \dots \leq e_n \leq e_1 + m +1.
\end{equation}
Then \eqref{eq:cellscover} holds for $\mathcal{M}(d,e,f)$. 
In particular, Proposition~\ref{prop:bbbowvar} and Corollary~\ref{cor:bbbowvarfam} hold for any $d$ and $f$.
\end{theorem}

For quiver varieties $e$ is an $m$-bounded non-decreasing sequence, which implies $m+1$-boundedness. Hence Theorem \ref{thm:cbbcelldec_cover} extends well-known properties of quiver varieties to a larger class of bow varieties.

\begin{remark} 
For our specific choice of the one-parameter subgroup $\gamma$, computer evidence suggests that \eqref{eq:weakquivcond} is not only sufficient but also necessary for \eqref{eq:cellscover} to hold for all $f\in \ZZ^m$ and~$d$. 
\end{remark}

\begin{proof}[{Proof of Theorem~\ref{thm:cbbcelldec_cover}}] 
We will use Theorem~\ref{thm:tangentcharm}; denote the two terms of $N_{\alpha,\beta}$ in it by $Q$ and $R$ respectively. By the same considerations as in the proof of Theorem~\ref{thm:poincare_zastava}, the dimension of a $(+)$-cell depends on the number summands in $N_{\alpha,\beta}$ such that the weight of $t_2$ is positive, resp. nonnegative when $\alpha < \beta$, resp. $\alpha > \beta$, while it is independent from the components of $e$ when $\alpha=\beta$.

Recall that if $e_1=\dots=e_n=0$, then the bow variety directly corresponds to a quiver variety. It is known that quiver varieties of affine type A always have a maximal dimensional cell \cite{nakajima2004quiver, fujii2017combinatorial}. Note as well that the shift $t_1^{e_{\beta}-e_{\alpha}}$ in the $Q$-term of $N_{\alpha,\beta}$ does not change the $t_2$-weights of the summands. Combining these facts, we get that for any bow variety there exists a fixed point for which all the $Q$-terms lie in the positive weight space for our choice of the one-parameter subgroup $\gamma$. 


Our assumption assures that 
\[ 
\begin{aligned}
0 & \leq  e_{\beta}- e_{\alpha} & \quad \textrm{if } \alpha \leq \beta \\
-m-1 & \leq  e_{\beta}- e_{\alpha} & \quad \textrm{if } \alpha > \beta \\
\end{aligned}
\]
for all $1 \leq \alpha,\beta \leq n$. By definitions \eqref{eq:Rdef} and \eqref{eq:Tdef}, these imply that all terms in $R_{\alpha,\beta}^{e_{\alpha},e_{\beta}}$ have positive, resp. nonnegative $t_2$-weight when $\alpha < \beta$, resp. $\alpha > \beta$.

Combining all these, we get that for the specific fixed point obtained above both the $Q$ and $R$ terms lie in the positive weight space. So the corresponding cell is maximal dimensional. In $\mathcal{M}_0$ the $(+)$-set whence covers the whole variety, while the $(-)$-set is a point.

\end{proof}

\subsubsection{Quiver-like bow varieties}
Observe that the favorable properties collected in Proposition~\ref{prop:bbbowvar} do not depend on the choice of the one-parameter subgroup $\gamma$. As long as a $\gamma$ exists with property \eqref{eq:cellscover} we have the 
consequences of Proposition \ref{prop:bbbowvar}. A simple way of considering another $\gamma$ for $\mathcal{M}(d,e,f)$ is carrying out a [move-1] on the triple $(d,e,f)$. Then the $\gamma$ defined by \eqref{eq:subgroup_choice} changes, because the torus gets reparametrized. Hence we obtain the following.

\begin{corollary} \label{cor:othersubgroup}
    If repeated applications of [move-1]'s (or its inverse) can change $(d,e,f)$ to a triple $(d',e',f')$ such that $e'$ is a $m+1$-bounded non-decreasing sequence, then $\mathcal{M}(d,e,f)$ satisfies the properties listed in Proposition~\ref{prop:bbbowvar}.
\end{corollary}

It is worth mentioning how to recognize vectors $e\in \Z^n$ that are [move-1]-equivalent to a $m+1$-bounded non-decreasing sequence. Define
\[
\delta_1=e_2-e_1, \ \delta_2=e_3-e_2, \ \ldots, \ \delta_{n-1}=e_n-e_{n-1}, \  \delta_n=(e_1+m)-e_n.
\]
\begin{proposition}
A vector $e\in \Z^n$ is [move-1]-equivalent to an $m+1$-bounded non-decreasing sequence if and only if 
\begin{equation}\label{eq:quiverlike}
\exists j \in \{1,\ldots,n\}: \qquad  \delta_j\geq -1 \text{ and } (\delta_i\geq 0 \text{ for } i\not= j).
 \end{equation}
\end{proposition}
\begin{proof} Follows from the definition of [move-1].
\end{proof}
This condition is a relaxation of the ``all $\delta_i\geq 0$'' condition that would make the $\mathcal{M}(d,e,f)$ variety a quiver variety, cf. Theorem \ref{wittentheorem}. Therefore the bow varieties $\mathcal{M}(d,e,f)$ for which $e$ satisfies \eqref{eq:quiverlike} can be considered \emph{quiver-like}: they share the crucial properties collected in Proposition~\ref{prop:bbbowvar} with quiver varieties.

\section{Partition function}
\label{sec:partition_function}

Recall that the equivariant homology group is a module over the usual equivariant cohomology of a
point $H^{\ast}_T(\pt)$. The latter is the symmetric algebra $S(\mathfrak{t}^{\ast})$ of the dual of the Lie algebra of $T$. Denote by $\varepsilon_1, \varepsilon_2, a_1,\dots, a_r$ its generators corresponding respectively to $t_1, t_2, \u_1,\dots, \u_r$, the generators of the equivariant K-theory. For short, we write  $a=(a_1,\dots, a_r)$. 

By the localisation theorem for equivariant homology there exists a commutative diagram
\begin{center}
\begin{tikzcd}
H_{\ast}^T(\mathcal{M}) \otimes_{S(\mathfrak{t}^{\ast})} \mathcal{S} \ar[r, "\cong", "(\iota_{\ast})^{-1}"'] \ar[d,"\pi_{\ast}"] & \oplus_M  \mathcal{S} \ar[d,"\sum_{M}"] \\
H_{\ast}^T(\mathcal{M}_0) \otimes_{S(\mathfrak{t}^{\ast})} \mathcal{S} \ar[r, "\cong", "(\iota_{0\ast})^{-1}"'] &   \mathcal{S}
\end{tikzcd}
\end{center}

Fix margin vectors $(e,f)$ such that $\mathcal{M}(d,e,f)$ satisfies \eqref{eq:cellscover} for all $d \geq 0$. We define the \emph{parabolic partition function} as the following generating function:
\[Z(\varepsilon_1,\varepsilon_2,a,q) = \sum_{d=0}^{\infty} Z_n(\varepsilon_1,\varepsilon_2,a)q^d =  \sum_{d=0}^{\infty} (\iota_{0\ast})^{-1} \pi_{\ast}[\mathcal{M}(d,e,f)]q^d.\]
\begin{proposition}
\label{prop:nekrasov_coeff}
\[ Z(\varepsilon_1,\varepsilon_2,a) = \sum_{M \in M(d,e,f)} \frac{1}{e(T_M)}\]
\end{proposition}
\begin{proof} 
Replacing the quiver variety appearing in \cite[Section~4]{nakajima2005instanton} by the bow variety $\mathcal{M}(d,e,f)$  we have that for any class $\alpha$
\[ (\iota_{0\ast})^{-1} \pi_{\ast}(\alpha) = \sum_M \frac{\iota_M^{\ast}\alpha}{e(T_M)}\]
holds. In particular,
\[ 
(\iota_{0\ast})^{-1} \pi_{\ast}[\mathcal{M}(d,e,f)] = \sum_{M} \frac{1}{e(T_M)}.\]
\end{proof}
\begin{proposition}
\[Z(\varepsilon_1,\varepsilon_2,a,q)=\sum_{M \in M(e,f)} \frac{{q}^{|M|}}{\prod_{\alpha,\beta=1}^n n_{\alpha,\beta}^M(\varepsilon_1,\varepsilon_2,a)}\]
where 
\[ 
\begin{multlined}
n_{\alpha,\beta}^M(\varepsilon_1,\varepsilon_2,a)=\prod_{s \in Y_{\alpha}} \left((e_{\beta}-e_{\alpha}-l_{Y_{\beta}}(s))\varepsilon_1 + (a_{Y_{\alpha}}(s)+1)\varepsilon_2+a_{\beta}-a_{\alpha}\right)^{\delta^{(m)}_{\u_{Y_{\beta},Y_{\alpha}}(s),e_{\alpha}-e_{\beta}}} \\ \times  \prod_{t \in Y_{\beta}}\left((e_{\beta}-e_{\alpha}+l_{Y_{\alpha}}(t)+1)\varepsilon_1-a_{Y_{\beta}}(t)\varepsilon_2+a_{\beta}-a_{\alpha}\right)^{\delta^{(m)}_{\u_{Y_{\alpha},Y_{\beta}}(s),e_{\beta}-e_{\alpha}}} \\ \times \prod_{(s_1,s_2) \in R^{e_{\alpha},e_{\beta}}_{\alpha,\beta}} (s_1\varepsilon_1+s_2\varepsilon_2+a_{\beta}-a_{\alpha})^{\delta_{s_1,s_2}^{(m)}}
\end{multlined}\]
\end{proposition}
\begin{proof} Follows by combining Theorem~\ref{thm:tangentcharm} and Proposition~\ref{prop:nekrasov_coeff}.
\end{proof}
Assume for simplicity that $m=1$. Denote by $Z^{M(r,n)}(\varepsilon_1,\varepsilon_2,a)$ Nekrasov's partition function \cite[Section~6]{nakajima2005instanton} for the Nakajima quiver variety $M(r,n)$. Note that the third term in the above expression for $n_{\alpha,\beta}^M(\varepsilon_1,\varepsilon_2,a)$ does not depend on $M$, but only on $e$. This observation leads us to the following.
\begin{corollary}
\label{cor:partfunasquivfn}
Let $m=1$. Assume that $e \in \ZZ^n$ is a $2$-bounded non-decreasing sequence. Then for any $d \in \mathbb{Z}_{\geq 0}$
\[Z(\varepsilon_1,\varepsilon_2,a )= Z^{M(r,n)}(\varepsilon_1,\varepsilon_2,a + \varepsilon_1 e) \cdot \prod_{\substack{1 \leq \alpha, \beta \leq n \\ (s_1,s_2) \in R^{e_{\alpha},e_{\beta}}_{\alpha,\beta}}} \frac{1}{ s_1\varepsilon_1+s_2\varepsilon_2+a_{\beta}-a_{\alpha}}.\]
Hence,
\[Z(\varepsilon_1,\varepsilon_2,a,q )= Z^{M(r,n)}(\varepsilon_1,\varepsilon_2,a + \varepsilon_1 e,q) \cdot \prod_{\substack{1 \leq \alpha, \beta \leq n \\ (s_1,s_2) \in R^{e_{\alpha},e_{\beta}}_{\alpha,\beta}}} \frac{1}{ s_1\varepsilon_1+s_2\varepsilon_2+a_{\beta}-a_{\alpha}}.\]
\end{corollary}
It would be exciting to construct a parabolic analogue of the moduli space of instantons on the blown-up plane and to investigate the counterpart of Nekrasov's regularity conjecture. As bow varieties are expected to parameterise parabolic sheaves on the plane blown-up already in one point, the above spaces could be related in fact to sheaves on the plane blown-up in two points. 

Beside observing their formal similarity, the precise relation between $Z(\varepsilon_1,\varepsilon_2,a,q)$ and the partition function of the blown-up plane \cite[(6.10)]{nakajima2005instanton} is another intriguing question we leave for further investigation.

\appendix

\section{Modularity}
\label{sec:modularity}
\begin{center}by \textsc{Gergely Harcos}\end{center}
\medskip

\begin{theorem}
\label{thm:mod}
Let $S$ be the set of those $3\times 3$ integral matrices $w=(w_{ij})$ which satisfy
\[\begin{pmatrix}1&1&1\end{pmatrix}w=\begin{pmatrix}1&1&1\end{pmatrix}w^T=\begin{pmatrix}3&2&1\end{pmatrix}.\]
Consider the quadratic polynomial
\[F(w):=\sum_{i,j}w_{ij}(w_{ij}-1)/2,\qquad w\in S,\]
and the generating function (cf. Example~\ref{example:321321})
\[Z_0(q):=\sum_{w\in S}q^{F(w)},\qquad|q|<1.\]
Then
\begin{equation}\label{eq2}
Z_0(q)=\sum_{n=0}^\infty\sigma(3n+1)q^n.
\end{equation}
\end{theorem}

\begin{proof}
We parametrize the integral matrices $w\in S$ by integral column vectors
\[m:=\begin{pmatrix}a&b&c&d\end{pmatrix}^T\in\ZZ^4\]
as follows:
\[w=\begin{pmatrix}
1&1&1\\1&1&0\\1&0&0
\end{pmatrix}+
\begin{pmatrix}
a&b&-a-b\\c&d&-c-d\\-a-c&-b-d&a+b+c+d
\end{pmatrix}.\]
With this parametrization, we can rewrite $F(w)$ as
\begin{equation}\label{eq1}
F(w)=\frac{m^TAm}{2}+d-a,
\end{equation}
where
\[A:=\begin{pmatrix}
4&2&2&1\\
2&4&1&2\\
2&1&4&2\\
1&2&2&4
\end{pmatrix}.\]
For later reference, we record that
\[\det(A)=81\qquad\text{and}\qquad 9A^{-1}=\begin{pmatrix*}[r]
4&-2&-2&1\\
-2&4&1&-2\\
-2&1&4&-2\\
1&-2&-2&4
\end{pmatrix*}.\]
In particular, the column vector
\[x:=A^{-1}\begin{pmatrix}-1&0&0&1\end{pmatrix}^T=\begin{pmatrix}-1/3&0&0&1/3\end{pmatrix}^T\]
satisfies
\begin{align*}
m^TAx&=\begin{pmatrix}a&b&c&d\end{pmatrix}\begin{pmatrix}-1&0&0&1\end{pmatrix}^T=d-a,\\
x^TAx&=\begin{pmatrix}-1/3&0&0&1/3\end{pmatrix}\begin{pmatrix}-1&0&0&1\end{pmatrix}^T=2/3,\\
\end{align*}
so that \eqref{eq1} becomes
\[F(w)=\frac{(m+x)^TA(m+x)}{2}-\frac{1}{3}.\]

As usual, we write $q=e(z):=e^{2\pi i z}$, where $\Im z>0$. Then, by the above calculation,
\[Z_0(e(z))=\sum_{m\in\ZZ^4}e\left(\frac{(m+x)^TA(m+x)}{2}z-\frac{1}{3}z\right)
=\sum_{v\in x+\ZZ^4}e\left(\frac{v^TAv}{2}z-\frac{1}{3}z\right).\]
It will be convenient to replace $v$ by $v/9$, and then multiply both sides by $e(z/3)$:
\[e(z/3)Z_0(e(z))=\sum_{v\in 9x+9\ZZ^4}e\left(\frac{v^TAv}{162}z\right).\]
Introducing
\[h:=9x=\begin{pmatrix}-3&0&0&3\end{pmatrix}^T,\]
this is just
\[e(z/3)Z_0(e(z))=\sum_{v\equiv h\pmod{9}}e\left(\frac{v^TAv}{162}z\right).\]
In the notation of Shimura~\cite[(2.0)]{shimura}, the right-hand side equals $\theta(z;h,A,9,1)$, which is apparently the same as $\theta(z;-h,A,9,1)$. Now \cite[Prop.~2.1]{shimura} and the subsequent remarks on \cite[p.~456]{shimura} show that this theta series is a modular form of weight $2$ for the congruence subgroup
\[\Gamma:=\left\{\begin{pmatrix}a&b\\c&d\end{pmatrix}\in\SL_2(\ZZ):\ b\equiv 0\pmod{3},\ c\equiv 0\pmod{9}\right\}.\]
We replace $z$ by $3z$, which has the same effect as replacing $\Gamma$ by its conjugate
\[\begin{pmatrix}1/3&0\\0&1\end{pmatrix}\Gamma\begin{pmatrix}3&0\\0&1\end{pmatrix}=\Gamma_0(27).\]
We conclude that $e(z)Z_0(e(3z))=q Z_0(q^3)$ is a modular form of weight $2$ and level $27$.

Finally, let us look at the Eisenstein series of weight $2$ and level $1$:
\[E(q):=-\frac{1}{24}+\sum_{m=1}^\infty\sigma(m)q^m,\qquad|q|<1.\]
By \cite[Ch.~III, Prop.~7]{koblitz}, the function $E(q)-3E(q^3)$ is a modular form of weight $2$ and level $3$. If we twist this modular form by the two mod $3$ Dirichlet characters, then by \cite[Ch.~III, Prop.~17]{koblitz}, we get two modular forms of weight $2$ and level $27$. The average of these two twisted modular forms equals
\[G(q):=\sum_{n=0}^\infty\sigma(3n+1)q^{3n+1},\qquad|q|<1.\]
By inspection, the coefficients of $G(q)$ agree with the coefficients of $q Z_0(q^3)$ up to degree $6$, which is the Sturm bound for weight $2$ and level $27$. This forces the identity
\[q Z_0(q^3)=G(q),\]
which is equivalent to \eqref{eq2}.
\end{proof}

\bibliographystyle{amsplain}
\bibliography{bow}

\providecommand{\bysame}{\leavevmode\hbox to3em{\hrulefill}\thinspace}
\providecommand{\MR}{\relax\ifhmode\unskip\space\fi MR }
\providecommand{\MRhref}[2]{%
  \href{http://www.ams.org/mathscinet-getitem?mr=#1}{#2}
}
\providecommand{\href}[2]{#2}
\begin{thebibliography}{10}

\bibitem{barvinok}
A.~Barvinok, \emph{Matrices with prescribed row and column sums}, Linear
  Algebra and its Applications \textbf{436} (2012), no.~4, 820--844.

\bibitem{bellamy2020birational}
G.~Bellamy and A.~Craw, \emph{Birational geometry of symplectic quotient
  singularities}, Inventiones mathematicae \textbf{222} (2020), no.~2,
  399--468.

\bibitem{bialynicki1973some}
A.~Bialynicki-Birula, \emph{Some theorems on actions of algebraic groups},
  Annals of Mathematics \textbf{98} (1973), no.~3, 480--497.

\bibitem{BottaRimanyi}
T.~M. Botta and R.~Rim\'anyi, \emph{Bow varieties: {S}table envelopes and their
  3d mirror symmetry}, preprint, arXiv:2308.07300.

\bibitem{cherkis2011instantons}
S.~Cherkis, \emph{{Instantons on gravitons}}, Communications in Mathematical
  Physics \textbf{306} (2011), 449--483.

\bibitem{foster2023tangent}
A.~Foster and Y.~Shou, \emph{Tangent weights and invariant curves in type {A}
  bow varieties}, preprint, arXiv:2310.04973 (2023).

\bibitem{fujii2017combinatorial}
Sh. Fujii and S.~Minabe, \emph{{A combinatorial study on quiver varieties}},
  SIGMA. Symmetry, Integrability and Geometry: Methods and Applications
  \textbf{13} (2017), 052.

\bibitem{gyenge2017enumeration}
{\'A}.~Gyenge, \emph{{Enumeration of diagonally colored Young diagrams}},
  Monatshefte f{\"u}r Mathematik \textbf{183} (2017), 143--157.

\bibitem{gyenge2018euler}
{\'A}.~Gyenge, A.~N{\'e}methi, and B.~Szendr{\H{o}}i, \emph{{Euler
  characteristics of Hilbert schemes of points on simple surface
  singularities}}, European Journal of Mathematics \textbf{4} (2018), 439--524.

\bibitem{OEIS}
OEIS~Foundation Inc., \emph{The on-line encyclopedia of integer sequences},
  2024, Published electronically at https://oeis.org.

\bibitem{ji2023bow}
Y.~Ji, \emph{{Bow varieties as symplectic reductions of $T^*(G/P)$}}, preprint
  arXiv:2312.04696 (2023).

\bibitem{kamnitzer2022symplectic}
J.~Kamnitzer, \emph{{Symplectic resolutions, symplectic duality, and Coulomb
  branches}}, Bulletin of the London Mathematical Society \textbf{54} (2022),
  no.~5, 1515--1551.

\bibitem{koblitz}
N.~Koblitz, \emph{Introduction to elliptic curves and modular forms}, GTM,
  no.~97, Springer-Verlag, 1984.

\bibitem{lauritzen}
S.~L. Lauritzen, \emph{Lectures on contingency tables}, Aalborg University,
  2002, 4th ed.

\bibitem{nagao2009quiver}
K.~Nagao, \emph{{Quiver varieties and Frenkel--Kac construction}}, Journal of
  Algebra \textbf{321} (2009), no.~12, 3764--3789.

\bibitem{nakajima2004quiver}
H.~Nakajima, \emph{{Quiver varieties and t-analogs of q-characters of quantum
  affine algebras}}, Annals of mathematics (2004), 1057--1097.

\bibitem{nakajima2018towards}
\bysame, \emph{{Towards geometric Satake correspondence for Kac-Moody
  algebras--Cherkis bow varieties and affine Lie algebras of type $ A$}}, arXiv
  preprint arXiv:1810.04293 (2018).

\bibitem{nakajima2017cherkis}
H.~Nakajima and Y.~Takayama, \emph{{Cherkis bow varieties and Coulomb branches
  of quiver gauge theories of affine type A}}, Selecta Mathematica \textbf{23}
  (2017), 2553--2633.

\bibitem{nakajima2004lectures}
H.~Nakajima and K.~Yoshioka, \emph{{Lectures on instanton counting}}, CRM
  Proceedings and Lecture Notes, vol.~38, American Mathematical Society, 2004,
  pp.~31--101.

\bibitem{nakajima2005instanton}
\bysame, \emph{{Instanton counting on blowup. I. 4-dimensional pure gauge
  theory}}, Inventiones mathematicae \textbf{162} (2005), no.~2, 313--355.

\bibitem{rimanyi2020bow}
R~Rim{\'a}nyi and Y~Shou, \emph{Bow varieties---geometry, combinatorics,
  characteristic classes}, preprint arXiv:2012.07814 (2022), to appear in Comm.
  in Analysis and Geometry.

\bibitem{shimura}
G.~Shimura, \emph{On modular forms of half integral weight}, Ann. of Math.
  \textbf{97} (1973), no.~2, 440--481.

\bibitem{witten}
E.~Witten, \emph{Branes, instantons, and {T}aub-{NUT} spaces}, Journal of High
  Energy Physics (2009), no.~06, 067.

\end{thebibliography}

\end{document}